\documentclass[10pt]{amsart}
\usepackage{amsmath}
\usepackage{amsfonts}
\usepackage{amssymb}
\usepackage{amsthm}
\usepackage{url}
\usepackage[all]{xy}
\usepackage{dsfont} %
\usepackage{graphicx}
\usepackage{caption}
\usepackage{subcaption}
\usepackage{comment} 
\usepackage{stmaryrd}%
\usepackage{hyperref}
\usepackage{todonotes}
\usepackage{color}
\usepackage{enumerate,tikz-cd}
\usepackage{fixltx2e}
\usepackage{MnSymbol}%
\usepackage{soul}
\usepackage{enumitem}
\allowdisplaybreaks

\numberwithin{equation}{section}

\newtheorem{proposition}{Proposition}[section]
\newtheorem{lemma}[proposition]{Lemma}
\newtheorem{theorem}[proposition]{Theorem}
\newtheorem{corollary}[proposition]{Corollary}

\theoremstyle{definition}
\newtheorem{remark}[proposition]{Remark}
\newtheorem{definition}[proposition]{Definition}
\newtheorem{example}[proposition]{Example}

\DeclareMathOperator{\End}{End}
\DeclareMathOperator{\Id}{Id}
\DeclareMathOperator{\tr}{tr}

\DeclareMathOperator{\Ric}{Ric}

\DeclareMathOperator{\V}{\mathcal{V}}
\DeclareMathOperator{\ch}{ch}
\DeclareMathOperator{\Td}{Td}
\DeclareMathOperator{\rk}{rk}

\DeclareMathOperator{\Ima}{Im}

\DeclareMathOperator{\Lie}{Lie}

\DeclareMathOperator{\eq}{eq}

\DeclareMathOperator{\Adj}{Adj}

\newcommand{\R}{\mathbb{R}}
\newcommand{\C}{\mathbb{C}}

\newcommand{\G}{\mathcal{G}}

\renewcommand{\epsilon}{\varepsilon}

\newcommand{\scA}{\mathcal{A}}

\newcommand{\scV}{\mathcal{V}}
\newcommand{\ddb}{i\partial \bar\partial}

\newcommand{\J}{\mathcal{J}}

\newcommand{\mfk}{\mathfrak{k}}

\newcommand{\E}{\mathcal{E}}
\renewcommand{\L}{\mathcal{L}}

\newcommand{\U}{\mathcal{U}}

\renewcommand{\G}{\mathcal{G}}

\newcommand{\ddbar}{\partial\overline{\partial}}
\renewcommand{\phi}{\varphi}

\renewcommand{\d}{\partial}
\newcommand{\db}{\bar\partial}

\renewcommand{\O}{\mathrm{O}}

\title[Moment maps in complex geometry]{ The universal structure of moment maps in complex geometry}

\author[Ruadha\'i Dervan and Michael Hallam]{Ruadha\'i Dervan and Michael Hallam}

\address{Ruadha\'i Dervan, School of Mathematics and Statistics, University of Glasgow, University Place, Glasgow G12 8QQ, United Kingdom}\email{ruadhai.dervan@glasgow.ac.uk}

\address{Michael Hallam, Department of Mathematics, University of Aarhus, Ny Munkegade 118, 8000 Aarhus C, Denmark}\email{hallam@math.au.dk}
\begin{document}

\begin{abstract} We introduce a geometric approach to the construction of moment maps in finite and infinite-dimensional complex geometry. We apply this to two settings: K\"ahler manifolds and holomorphic vector bundles. Our new approach exploits the existence of universal families and the theory of equivariant differential forms.

We first give a new, geometric proof of Donaldson--Fujiki's moment map interpretation of the scalar curvature. Associated to arbitrary products of Chern characters of the manifold---namely to a central charge---we  further introduce a geometric PDE determining a $Z$-critical K\"ahler metric, and show that these general equations also satisfy moment map properties. For holomorphic vector bundles, using a similar strategy  we  give a geometric proof of Atiyah--Bott's moment map interpretation of the Hermitian Yang--Mills condition. We then go on to give a new, geometric proof that the PDE determining a $Z$-critical connection---again associated to a choice of central charge---can be viewed as  a moment map; deformed Hermitian Yang--Mills connections are a special case, in which our work gives a geometric proof of a result of Collins--Yau.

Our main assertion is that this is the canonical way of producing moment maps in complex geometry---associated to any geometric problem along with a choice of stability condition---and hence that this accomplishes one of the main steps towards producing PDE counterparts to stability conditions in large generality.
\end{abstract}

\maketitle

\section{Introduction}

Many of the most important geometric PDEs in complex geometry can be viewed as moment maps. The two most prominent such examples are the \emph{constant scalar curvature K\"ahler (cscK) equation}, which through Donaldson--Fujiki arises as a moment map on the space of almost complex structures on a compact symplectic manifold  \cite{SD-moment,AF}, and the \emph{Hermitian Yang--Mills equation}, which through Atiyah--Bott arises as a moment map on the space of unitary connections on a Hermitian vector bundle \cite{AB}. These moment map properties are ultimately behind the deep links between these PDEs and algebro-geometric stability conditions (through the Yau--Tian--Donaldson conjecture \cite{STY,GT,SD-toric} and the Hitchin--Kobayashi correspondence of Donaldson--Uhlenbeck--Yau  \cite{SD-surfaces,SD-infinite,UY} respectively).

In algebraic geometry, it has become increasingly important to consider general classes of stability conditions  \cite{TB}. It is thus desirable to associate  geometric PDEs to  general stability conditions, in such a way that solvability of these PDEs is equivalent to stability. The fundamental issue is that the classical PDEs of interest---notably the cscK equation and the Hermitian Yang--Mills equation---arose from much more classical geometric considerations, with their moment map properties being proven far later. The proofs that these (and other) equations satisfy moment map properties are usually by direct calculation, and do not quite explain precisely why these equations arise as moment maps. Thus in aiming to generate general classes of geometric PDEs associated with algebro-geometric stability conditions, a new approach is needed. 

What we achieve here is as follows: \begin{enumerate}[label={(\roman*)}]
\item We give a new, geometric approach to the construction of moment maps in large generality. Through our approach, moment maps arise whenever one has an equivariant family of objects (such as a family of complex manifolds or holomorphic vector bundles). The input needed to introduce a moment map is essentially the same as that needed to define an algebro-geometric stability condition in a given setting. Thus from (essentially) the \emph{same} input, we obtain both a geometric PDE and a stability condition, generating automatic links between analysis and algebraic geometry.
\item We use this approach to derive a new general class of geometric PDEs for K\"ahler metrics on complex manifolds. A special case thus gives a geometric derivation of the cscK equation. On holomorphic vector bundles, we give a new proof that the Hermitian Yang--Mills equation is a moment map; our proof extends to the deformed Hermitian Yang--Mills equation (first proven by Collins--Yau in rank one \cite{CY})---giving a geometric derivation of this equation---and further to the  $Z$-critical connection equation \cite{DMS}. 
\item Our technique applies in more generality than prior approaches. The ideas apply both to finite-dimensional and infinite-dimensional situations, and allow simultaneous variation of both complex structure and metric structure (this is important for technical applications of our moment map results \cite{ortu-sektnan,ortu, DS-AP}). \end{enumerate}

\subsection*{Main results for manifolds.} Our approach is geometric, and quite simple. We primarily discuss the finite-dimensional case, though the argument in the infinite-dimensional case is similar. The setup involves a proper holomorphic submersion $\pi: (X,\alpha) \to B$, with $X$ and $B$ complex manifolds and with $\alpha$ a class on $X$ which is K\"ahler on each fibre of $\pi$. We assume that a compact Lie group $K$ acts on $X$ and $B$ by biholomorphisms such that $\pi$ is a $K$-equivariant map, and fix a $K$-invariant form $\omega \in \alpha$ which is K\"ahler on each fibre. This relatively K\"ahler metric $\omega$ induces a Hermitian metric on the relative anticanonical class $-K_{X/B}$ of $X\to B$, with curvature (times \(\frac{i}{2\pi}\)) which we denote by $\rho \in c_1(-K_{X/B})$. The form $\rho$ can be viewed as a relative analogue of the Ricci curvature, and on each fibre restricts to the Ricci curvature of the restriction of $\omega$.

We then invoke the classical theory of equivariant differential forms.  Equivariant differential geometry has been used in the theory of constant scalar curvature K\"ahler metrics since at least the work of Futaki--Morita \cite{futaki-mabuchi}; we learned the theory through the more recent work of Legendre, Inoue and Corradini \cite{EL, EI, AC}. We assume that the \(K\)-action is \(\omega\)-Hamiltonian, so that there exists a moment map $\mu: X \to \mfk^* = (\Lie K)^*$ for the $K$-action on $(X,\omega)$ (where here and throughout we do not demand positivity of $\omega$ in the definition of a moment map). The equivariant differential form defined by sending $v\in \mfk$ to $\omega + \langle \mu,v\rangle$ is then equivariantly closed. We show that in turn this implies that the equivariant differential form $$v \mapsto \rho - \frac{1}{2\pi}\Delta_{\V}\langle \mu, v\rangle$$ is equivariantly closed, where $\Delta_{\V}$ is the vertical Laplacian (i.e. calculated fibrewise). This then implies that the equivariant form $\eta: \mathfrak{k} \to \Omega^*(X)$ defined by the wedge product $$\eta(v) :=  \frac{\hat S_b}{n+1}(\omega + \langle \mu,v\rangle)^{n+1} - \left(\rho - \frac{1}{2\pi}\Delta_{\V}\langle \mu, v\rangle\right)\wedge (\omega + \langle \mu,v\rangle)^n $$ is equivariantly closed, where $n:=\dim X - \dim B$, and $\hat S_b$ is the topological constant defining the fibrewise average scalar curvature. Since $\pi: X\to B$ is $K$-equivariant, it follows that the fibre integral of this form is itself equivariantly closed. For each $v\in\mathfrak{k}$, for degree reasons, this form is the sum of a $(1,1)$-form \(\Omega\) and a function \(\sigma_v\) on $B$. A direct, simple calculation shows that the function is defined by $$\sigma_v(b) = \int_{X_b} \langle\mu,v\rangle|_{X_b}(\hat S_b - S(\omega_b)) \omega_b^n,$$ where $\omega_b$ denotes the restriction of $\omega$ to $X_b$; the calculation involves using that the integral of a function in the image of the Laplacian vanishes.  The $(1,1)$-form on $B$ is $$\Omega = \frac{\hat S_b}{n+1}\int_{X/B}\omega^{n+1} -\int_{X/B} \rho \wedge \omega^n,$$ and that the equivariant differential form \(v\mapsto \Omega + \sigma_v\) on $B$ is equivariantly closed is precisely the moment map condition for the scalar curvature with respect to $\Omega$. 

Thus we obtain the moment map property for the scalar curvature through equivariant differential geometry and the theory of equivariant holomorphic submersions. The form $\Omega$ is the standard Weil--Petersson form involved in the cscK theory, and while it is not K\"ahler in general (which we permit in our definition of a moment map), it is K\"ahler in many situations of interest and its positivity is well-understood \cite{FS-moduli, AF}. 

One can view this as a \emph{derivation} of the scalar curvature. From this perspective, it becomes clear what the choice involved was: it was exactly the choice of equivariant form $\eta$ on $X$. We view this choice as arising from one copy of the first Chern class (through the involvement of $\rho$) and a complementary number of copies of $\alpha$ (through the involvement of $\omega$) so that the total degree is correct. The most general input thus involves products of \emph{arbitrary} Chern classes of correct total degree, and this is encoded in the notion of a  \emph{central charge} $Z$, which is ultimately a choice of topological input. Through our moment map construction, we then introduce the notion of a $Z$\emph{-critical K\"ahler metric}, solving a PDE taking the form (on a fibre $(X_b,\omega_b)$) $$\Ima(e^{-i\phi} \tilde Z(X_b, \omega_b)) = 0,$$ associated to the central charge $Z$, where $\tilde Z(X_b, \omega_b)$ is a complex-valued function involving various curvature quantities associated to the metric and $\phi \in (-\pi,\pi]$. This equation had been derived in \cite{stabilityconditions} in the case when only products of the first Chern class are used, and our new approach allows us to derive the correct equation in general.

The general, explicit form of the PDE is rather intricate---see Definition \ref{defpde}---and involves more than just Chern--Weil theory (essentially the coincidence that the ``vertical Laplacian term'' vanishes in the derivation of the scalar curvature as a moment map does not happen in general, complicating matters significantly). This leads to an equation that is a \emph{sixth-order} PDE in the K\"ahler potential, in comparison with the \emph{fourth-order} cscK equation. We emphasise, however, that our derivation of the equation (and hence the PDE itself) is completely geometric and canonical, starting from the choice of topological input. Our usage of complex-valued central charges, and imaginary parts of complex-valued functions is essentially an aesthetic choice that matches what is well-established in the Bridgeland stability literature, and does not play a significant role in our work (but is fundamental to the structure of the ``stability manifold'' in the theory of Bridgeland stability conditions \cite{TB}).

In the algebraic direction, through \cite{stabilityconditions, git-stabilityconditions}, a central charge canonically induces a notion of $Z$\emph{-stability}, generalising the notion of K-stability involved in the Yau--Tian--Donaldson conjecture. Thus, crucially, the same topological input---a central charge---induces in a canonical way both a geometric PDE and a stability condition.

We summarise our first main result as follows:

\begin{theorem}\label{introthm1}
The $Z$-critical K\"ahler operator arises as a moment map, in both finite and infinite dimensions.
\end{theorem}

The precise statement is given in Theorem \ref{thm:fin_dim_manifold_moment_map} in finite dimensions and Theorem \ref{thm:infinite_dim_manifold_moment_map} in infinite dimensions. The main point is to calculate the Chern--Weil representatives of the equivariant Chern characters of the relative tangent bundle of $X \to B$, defined through the Hermitian metric induced by the relatively K\"ahler metric $\omega$. In infinite dimensions, we consider the space $\J(M,\omega)$ of  almost complex structures on a fixed compact symplectic manifold, which is the setting of the Donaldson--Fujiki moment map interpretation of the scalar curvature \cite{SD-moment,AF} (which we thus give a new geometric proof of). We employ the universal family over this space, which is an almost holomorphic submersion of finite relative dimension, meaning that the geometric setting is directly analogous to that of the finite-dimensional one, and similar techniques apply to give moment map properties also in this setting. Here, the key is again to calculate equivariant Chern--Weil representatives in this almost K\"ahler setting.

Another consequence of our work, for which we expect future applications, is the production of a great many equivariantly closed differential forms on $\J(M,\omega)$ (and more generally the base of any equivariant family), which are equivariant representatives of natural, geometric cohomology classes.

\subsection*{Main results for holomorphic vector bundles} The setup and results are very similar in the setting of vector bundles. Here the notion of a central charge is a basic part of the input into the notion of a \emph{Bridgeland stability condition}, and involves certain products of Chern characters of the vector bundle itself. As a conjectural  model analytic counterpart to Bridgeland stability conditions, in \cite{DMS} the notion of a \emph{$Z$-critical connection} was introduced, as a geometric PDE for a connection $A$ taking the form $$\Ima(e^{-i\phi} \tilde Z(E, A)) = 0,$$ where $\tilde Z(E, A)$ is a complex $\End E$-valued $(n,n)$-form and $\Ima$ denotes (\(-i\) times) the skew-Hermitian part of the endomorphism (see Definition \ref{def:z-connection}).

The $Z$-critical equation is a generalisation of the \emph{deformed Hermitian Yang--Mills equation}, which is the most important special case (beyond the classical Hermitian Yang--Mills equation, to which our results also apply). Under mirror symmetry, this equation was introduced as the mirror of the \emph{special Lagrangian equation} by Leung--Yau--Zaslow \cite{LYZ}; there is also a direct derivation from string theory \cite{MMMS}. The deformed Hermitian Yang--Mills equation is simply the $Z$-critical equation for a special choice of central charge. There is an enormous body of important work on this equation on line bundles, due to  Collins, Jacob and Yau \cite{JY, CJY,CY}, Chen \cite{GC} and Datar--Pingali \cite{DP} (see also Song \cite{JS}), but relatively little is known in higher rank.

We fix a compact K{\"a}hler manifold \((X, \omega)\) and consider a family of holomorphic vector bundles on \(X\) parametrised by a base $B$ (which may not be compact); that is, a holomorphic vector bundle $\pi: E \to B \times X$. We endow $E$ with a metric and compatible connection, and assume there is a $K$-action on $B$ with a suitable lift to $E$. Our main result in this setting then canonically constructs the $Z$-critical equation as a moment map on $B$, again using equivariant Chern--Weil theory. In infinite dimensions, we consider the space of unitary connections over a fixed Hermitian vector bundle, as in Atiyah--Bott \cite{AB}. We show this space admits a universal vector bundle with a universal connection, linking with the finite-dimensional picture, and use similar geometric ideas to prove:

\begin{theorem}\label{introthm2}
The $Z$-critical connection operator arises as a moment map, in both finite and infinite dimensions.
\end{theorem}

We refer to Theorems \ref{connection-thm-infinite} and \ref{connection-thm-finite} for precise statements in infinite and finite dimensions respectively. In the infinite-dimensional setting, this gives a new geometric proof of \cite[Theorem 1.3]{DMS}, which in turn extended the moment map interpretation of the deformed Hermitian Yang--Mills operator due to Collins--Yau in rank one \cite[Section 2]{CY}. In particular, our work gives a geometric derivation of the deformed Hermitian Yang--Mills equation, starting from a central charge. We also note that our work produces a great many geometrically interesting equivariantly closed differential forms on the base of an equivariant family of vector bundles, such as on the space of unitary connections on a fixed Hermitian vector bundle.

Throughout our work, our moment map properties do not demand that the closed two-form involved in the definition is actually symplectic (i.e. nondegenerate or positive). In the vector bundle setting, positivity is understood through the  notion of a ``$Z$-subsolution'' \cite[Definition 2.33]{DMS}, which is an explicit equivalent condition to positivity of the relevant two-form (at a single point).   In turn, these subsolution conditions are related to algebro-geometric stability with respect to \emph{subvarieties} \cite[Theorems 4.3.13, 4.3.17]{JM}, in some sense explaining the lack of appearance of subsolutions in the Hermitian Yang--Mills setting (where saturated subsheaves are sufficient to test stability), and further imply that the $Z$-critical operator is elliptic \cite[Lemma 2.38]{DMS}. It is worth emphasising that the deformed Hermitian Yang--Mills theory has been very successful \emph{without} the two-form relevant to the moment map property being \emph{globally} positive, with subsolution conditions (providing positivity \emph{locally} near solutions) instead in some sense replacing this positivity in the theory.  We expect there to be an analogous subsolution theory that applies also to the manifold setting, and more generally.

\subsection*{Outlook} The shared structure between the manifold and the vector bundle settings is the use of \emph{families of objects}; either families of complex manifolds or families of holomorphic vector bundles. Given an equivariant family of objects, we construct moment maps once we have the extra structure of a metric, in either setting. The topological input is essentially a choice of a collection of Chern classes in both settings, while the differential-geometric input is essentially equivariant Chern--Weil theory. With these ingredients, the recipe to construct moment maps is almost identical in both settings.

We thus expect that this is the right general approach to constructing moment maps, in any geometric problem. In algebraic geometry, the point of stability conditions is to form moduli spaces parametrising stable objects, and in particular the usage of universal families is ubiquitous in moduli theory. Our approach can be thought of as a differential-geometric analogue of key aspects of that theory; we make great use of the differential geometry of \emph{families} of objects. We expect our strategy to further apply to many other settings, though leave this for future work.

We end with a more speculative point. The theory of Bridgeland stability conditions is axiomatic in nature: Bridgeland produces axioms that must be satisfied to define a stability condition on a triangulated category \cite{TB}. Thus to match that theory in differential geometry, one needs universal structures that apply to all geometric problems. This is loosely what is achieved by our approach: given a suitable notion of a family of objects (as is needed in the algebraic theory to produce moduli spaces of stable objects), and given a suitable notion of a metric, we expect our approach to be the one that leads to moment maps in a completely canonical way. We mention that this is related to goals of the (independent, work-in-progress) programme of Haiden--Katzarkov--Kontsevich--Pandit, which they title ``categorical K\"ahler geometry''; their (similar) goal is to produce axiomatic notions of metrics on objects in triangulated categories, in such a way that Bridgeland stability corresponds to existence of solutions to certain equations in the space of metrics attached to the object (defined axiomatically). What is provided here is a geometric recipe that in practice  produces such moment maps.

\subsection*{Outline} In Section \ref{sec2} we introduce the notion of a $Z$-critical K\"ahler metric and recall the notion of a $Z$-critical connection, each associated to a central charge in the relevant setting. We proceed in Section \ref{sec3} to recall the basic theory of equivariant differential forms, and equivariant Chern--Weil theory. Sections \ref{sec4} and \ref{sec:ac-structures} then prove our main moment map results in the context of complex manifolds, including Theorem \ref{introthm1} in the finite and infinite-dimensional settings, respectively. Section \ref{sec:Z-critical-connections} proves the analogous moment map results in the setting of holomorphic vector bundles, including Theorem  \ref{introthm2}.

\subsection*{Acknowledgements} The first author thanks Eveline Legendre and Alexia Corradini for explaining much of the theory of equivariant differential geometry to him and for many discussions that motivated the present work. The authors also thank Frances Kirwan, L\'aszl\'o Lempert, Johan Martens, John McCarthy,  Carlo Scarpa, Lars Sektnan and Jacopo Stoppa  for helpful conversations and for their interest in this work, and thank the referees for their suggestions. RD was funded by a Royal Society University Research Fellowship for the duration of this work.

\section{Geometric PDEs in complex geometry}\label{sec2}

\subsection{$Z$-critical K\"ahler metrics} Let $X$ be a compact complex manifold of dimension $n$ and let $\alpha \in H^{1,1}(X,\R)$ be a K\"ahler class; the algebro-geometrically inclined reader should consider the special case $\alpha = c_1(L)$ for an ample line bundle $L$, in which case $X$ is a smooth projective variety. To a K\"ahler metric $\omega \in \alpha$ we associate various curvature quantities, which are mostly defined in the following manner. A K\"ahler metric is in particular a Hermitian metric on the holomorphic tangent bundle $TX^{1,0}$, and so to $\omega$ we associate the curvature of the Chern connection on $TX^{1,0}$ induced by $\omega$; we denote the curvature of this connection by $R \in \scA^{1,1}(\End TX^{1,0}).$

One associates to a connection  \emph{Chern--Weil representatives} of the Chern characters, defined by $$\widetilde \ch_k(X,\omega)  :=  \tr\left(\frac{1}{k!}\left(\frac{i}{2\pi}R\right)^k\right).$$ Thus the $\widetilde \ch_k(X,\omega) \in \ch_k(X)$ represent the Chern characters of $X$. We also use the notation $$\Ric\omega = \widetilde \ch_1(X,\omega)$$  for the representative of the first Chern character, as it coincides with the \emph{Ricci curvature} of $\omega$.  As the first Chern character agrees with the first Chern class, we often use $c_1(X)$ rather than the equivalent $\ch_1(X)$.

\begin{definition}\
A K\"ahler metric $\omega\in\alpha$ is said to be a \emph{constant scalar curvature K\"ahler metric} if its \emph{scalar curvature} $$S(\omega) := \Lambda_{\omega} \Ric \omega$$ is constant.
\end{definition}

If a K\"ahler metric $\omega\in\alpha$ has constant scalar curvature, a simple integration argument using that $\Ric\omega \in c_1(X)$ implies that the resulting constant $\hat S$ must satisfy $$\hat S = \frac{n\int_Xc_1(X)\cdot \alpha^{n-1}}{\int_X\alpha^n},$$ where we can view the numerator and denominator variously as integrals of closed differential forms, cup products in cohomology or in the projective case as intersection numbers. We view this equation as being associated with the numbers $\int_Xc_1(X)\cdot \alpha^{n-1}$ and $\int_X\alpha^n$ and seek to understand the appropriate equation for other topological choices. To do so, we begin with the topological choice itself, through the notion of a central charge.

\begin{definition}\label{def:central-charge-manifolds}A \emph{central charge} associates to each pair \((X,\alpha)\) a sum $$Z(X,\alpha) = \sum_{j,k} a_{jk}\int_X\alpha^{j}\cdot \ch_{k_1}(X)\cdot\ldots\cdot \ch_{k_r}(X)\in \C,$$ where: 
\begin{enumerate}[label={(\roman*)}]
\item \(k=(k_1,\ldots,k_r)\) is a multi-index of arbitrary length $r \leq n$;
\item $a_{jk} \in \C$ are a collection of complex numbers;
\item the integrand is of total degree $2n$, so $j+k_1+\ldots+k_r = n$;
\item the \emph{phase} $\phi(X,\alpha) := \arg Z(X,\alpha) \in (-\pi,\pi]$ is well-defined, so $Z(X,\alpha) \neq 0$.
\end{enumerate}
\end{definition}

We view this as the most general topological input that should determine both a stability condition for $(X,\alpha)$ and a geometric PDE; we will not discuss stability  in the current work, but rather exclusively the analytic aspects. We refer to \cite[Sections 2.1, 4.1]{stabilityconditions} for a discussion of associated stability conditions, and \cite[Section 2.3]{git-stabilityconditions} for an axiomatic, stacky approach. Here, it is sufficient to consider a central charge as simply a choice of topological input.

\begin{example}An interesting choice of central charge takes the form$$Z(X,\alpha) = \int_X e^{-i\alpha}\cdot \ch(X),$$ with $\ch(X)=1+\ch_1(X)+\ldots+\ch_n(X)$ the total Chern character of $X$; this is analogous to the most standard central charge on the category of coherent sheaves mentioned in Example \ref{string-theory-central-charge} below, which arises naturally in mirror symmetry.
\end{example}

We associate to each term in the central charge an analytic counterpart which is a sum of \emph{two} terms; the first uses traditional Chern--Weil theory, while the second is more involved and uses a sort-of linearised version of Chern--Weil theory. Consider again the curvature $R \in \scA^{1,1}(\End TX^{1,0})$ associated to the K\"ahler metric $\omega$ on $X$. 

The process will be linear in the terms comprising the central charge, and so we first consider a single term taking the form $ \alpha^{j}\cdot \ch_{k_1}(X)\cdot\ldots\cdot \ch_{k_r}(X)$. To this we firstly associate the complex-valued function \begin{equation}\label{eq:first-term}
\frac{\omega^j \wedge \widetilde \ch_{k_1}(X,\omega)\wedge\ldots \wedge \widetilde{\ch}_{k_r}(X,\omega)}{\omega^n},
\end{equation} namely by taking the natural Chern--Weil representatives. To define the second function, we firstly denote for \(1\leq m \leq r\), $$\tilde \ell_m(X,\omega) := \frac{1}{j+1}\frac{\omega^{j+1}\wedge \widetilde \ch_{k_1}(X, \omega)\wedge\ldots\wedge \frac{1}{(k_m-1)!}\left(\frac{i}{2\pi}R\right)^{k_m-1} \wedge \ldots\wedge\widetilde \ch_{k_r}(X,\omega)}{\omega^n},$$ which is a (complex) section of $\End TX^{1,0}$ produced by replacing \(\widetilde{\ch}_{k_m}(X, \omega)\) with the \(\End TX^{1,0}\)-valued form \(\frac{1}{(k_m-1)!}(\frac{i}{2\pi}R)^{k_m-1}\) whose trace is \(\widetilde{\ch}_{k_m-1}(X, \omega)\), and including one extra copy of $\omega$. For a general element $$A \in \scA^0(\End TX^{1,0}) \cong \scA^0(TX^{1,0}\otimes T^*X^{1,0}),$$ we denote by $A^{\flat}$ the associated section of $T^*X^{0,1} \otimes T^*X^{1,0}$ defined via the Hermitian metric induced by $\omega$. The second function is then \begin{equation}\label{eq:second-term}
\sum_{m=1}^r \frac{-1}{2\pi}\d^*\db^* \left(\tilde \ell_m(X,\omega)^{\flat}\right),
\end{equation} which is again a complex-valued function. Summing the two terms \eqref{eq:first-term} and \eqref{eq:second-term} produces a function $\tilde f(\alpha^{j}\cdot \ch_{k_1}(X)\cdot\ldots\cdot \ch_{k_r}(X); \omega)$ which satisfies $$\int_X \tilde f(\alpha^{j}\cdot \ch_{k_1}(X)\cdot\ldots\cdot \ch_{k_r}(X); \omega) \omega^n =  \alpha^{j}\cdot \ch_{k_1}(X)\cdot\ldots\cdot \ch_{k_r}(X),$$ since the integral of each term involving an adjoint term vanishes. Adding each such term produces a function \begin{equation}\label{eq:Z-tilde-def}
\tilde Z(X,\omega) = \sum_{j,k} a_{jk} \tilde{f}(\alpha^{j}\cdot \ch_{k_1}(X)\cdot\ldots\cdot \ch_{k_r}(X); \omega) \in C^{\infty}(X,\C)
\end{equation} which satisfies $$\int_X \tilde Z(X,\omega) \omega^n = Z(X,\alpha).$$

\begin{definition}\label{defpde} We say that $\omega$ is a \emph{$Z$-critical K\"ahler metric} if $$\Ima(e^{-i\phi(X,\alpha)} \tilde Z(X,\omega)) = 0.$$
\end{definition}

We view this as a PDE on the space of K\"ahler metrics $\omega \in \alpha$. Varying the K\"ahler metric $\omega \to \omega+i\ddbar \psi,$ the equation becomes a sixth-order PDE in $\psi$ in general. 

\begin{remark} If one uses Chern \emph{classes} rather than Chern \emph{characters}, one can similarly express the resulting geometric PDE, which, although equivalent, is notationally more cumbersome. Details are provided in Remark \ref{rmk:c-vs-ch}.\end{remark}

\begin{example} Suppose only powers of $c_1(X)$ are involved. Then this recipe associates to a term  of the form $ \alpha^{j}\cdot c_{1}(X)^{n-j}$ the function \[\frac{\omega^j\wedge(\Ric\omega)^{n-j}}{\omega^n} - \frac{n-j}{j+1}\frac{1}{2\pi}\d^*\db^*\left(\frac{\omega^{j+1}\wedge(\Ric\omega)^{n-j-1}}{\omega^n}(\Id_{TX^{1,0}})^\flat\right). \]To compute the right-hand term, note that the flattened identity morphism is the metric \(g\). The antisymmetrisation map \(T^*X^{0,1}\otimes T^*X^{1,0}\to T^*X^{0,1}\wedge T^*X^{1,0}\) is an isomorphism compatible with the del-bar operators \(\db:\scA^0(T^*X^{1,0})\to \scA^{0,1}( T^*X^{1,0})\) and \(\db:\Omega^{1,0}(X)\to\Omega^{1,1}(X)\), and the image of \(g\) under antisymmetrisation is \(i\omega\). By the K{\"a}hler identities, \([\db^*, \wedge\omega] = i\d\), hence the expression above becomes \begin{align*}
& \frac{\omega^j\wedge\Ric\omega^{n-j}}{\omega^n} + \frac{1}{2\pi}\frac{n-j}{j+1}\d^*\d\left( \frac{\omega^{j+1}\wedge\Ric\omega^{n-j-1}}{\omega^n}\right) \\
=&\frac{\omega^j\wedge\Ric\omega^{n-j}}{\omega^n} - \frac{1}{2\pi}\frac{n-j}{j+1}\Delta\left( \frac{\omega^{j+1}\wedge\Ric\omega^{n-j-1}}{\omega^n}\right),
\end{align*}where $\Delta=-\d^*\d$ is the K{\"a}hler Laplacian defined using $\omega$. The resulting PDE was introduced in \cite[Section 2.1]{stabilityconditions} in this special case where only the first Chern character is involved\footnote{The corresponding equation given in \cite[Section 2.1]{stabilityconditions} is erroneously missing a factor of $1/2\pi$, ultimately due to a mistake on \cite[p. 21]{stabilityconditions}, where the incorrect expression $\Ric(\omega+\ddb \phi) - \Ric(\omega) = -\ddb \log{\frac{(\omega+\ddb \phi)^n}{\omega^n}}$ is used.}; the equation is new to the present work once higher Chern characters are included.

 Specialising further, taking $Z(X,\alpha) =  i\int_X \alpha^n -  \int_X c_1(X)\cdot \alpha^{n-1}$ produces the cscK equation (which is \emph{fourth}-order, as the Laplacian term vanishes).  We expect our work to also encompass fourth-order PDEs considered by Futaki and Leung defined using Chern--Weil theory \cite{futaki-perturbed, leung} (see also Bando \cite{bando}), related to terms such as $\int_X c_j(X)\cdot \alpha^{n-j}$, with $c_j(X)$ the $j$\textsuperscript{th}-Chern class.

 \end{example}

\subsection{$Z$-critical connections} Our recipe is similar, but easier in the setting of a holomorphic vector bundle, where we produce $Z$-\emph{critical connections}. Here the theory originates in \cite{DMS}, building on the important special case of \emph{deformed Hermitian Yang--Mills connections} \cite{LYZ,MMMS,CYarXiv}. We consider a holomorphic vector bundle $E$ over a compact K\"ahler manifold $(X,\omega)$ of dimension $n$ with $\omega \in \alpha$; a Hermitian metric $h$ on $E$ then determines a Chern connection $A$ with curvature $F_A \in \scA^{1,1}(\End E).$

\begin{definition} The connection $A$ induced by a Hermitian metric $h$ is a \emph{Hermitian Yang--Mills connection} if $$\frac{i}{2\pi}\Lambda_{\omega}F_A = \lambda \Id_E.$$

\end{definition}

The constant $\lambda$ is topological, given by $\lambda = n\frac{\deg E}{\rk E\int_X\alpha^n},$ where $\deg E := \int_X c_1(E)\cdot \alpha^{n-1}$. We again view the Hermitian Yang--Mills condition as being ``induced'' by $\deg E$ and $\rk E$ and ask for a more general geometric PDE associated to an arbitrary topological choice. Once more, the first step is to define the topological input:

\begin{definition}A \emph{central charge} is a group homomorphism $Z: K(X) \to \C$ from the Grothendieck group of $X$ to the complex numbers, taking the form $$Z(E) = \sum_{j=0}^n \int_X \rho_j\alpha^j\cdot\ch(E)\cdot \Theta,$$ where:
\begin{enumerate}[label={(\roman*)}]
\item $\ch(E)$ is the total Chern character of $E$;
\item $\rho_j \in \C$ are complex numbers;
\item $\Theta \in \bigoplus_k H^{k,k}(X,\C)$ is an auxiliary cohomology class, which may equal 1;
\item only the $2n$-degree component of the integrand is considered in the integral;
\item the \emph{phase} $\phi(E) := \arg Z(E) \in (-\pi,\pi]$ is well-defined, so $Z(E) \neq 0$.
\end{enumerate}
\end{definition}

\begin{example}\label{string-theory-central-charge} Two central charges arising from string theory and mirror symmetry are the standard examples: the first takes the form $$Z(E) = \int_X e^{-i\alpha}\cdot e^{-\beta}\cdot \ch(E),$$ with $\beta \in H^{1,1}(X,\R)$ a ``$B$-field'', while the second takes the form $$Z(E) =  \int_X e^{-i\alpha}\cdot e^{-\beta}\cdot \ch(E)\cdot \sqrt{\Td(X),}$$ where $\Td(X)$ denotes the Todd class of $X$. In these cases, the auxiliary class $\Theta$ is taken to induce the appearance of the $B$-field and of the square root of the Todd class. 

 \end{example}

We next use our connection and our K\"ahler metric to associate a geometric PDE to a central charge, through Chern--Weil representatives. For this we also fix a representative $\theta \in \Theta$, so that $\theta$ is a direct sum of closed differential forms (possibly of varying degrees) representing the closed cohomology class $\Theta$. 

The central charge is a linear combination of terms of the form $\int_X \alpha^{j}\cdot \ch_{k}(E)\cdot \Theta_{n-j-k},$ where $\Theta_{n-j-k} \in H^{n-j-k,n-j-k}(X,\C)$ is of degree $2(n-j-k)$ so that the integrand is of total degree $2n$. We associate to this term the $\End(E)$-valued complex $(n,n)$-form $$\omega^j\wedge\left(\frac{1}{k!}\left(\frac{i}{2\pi}F_A\right)^k\right) \wedge\theta_{n-j-k} \in \scA_{\C}^{n,n}(\End E),$$ and summing up produces an $\End E$-valued complex $(n,n)$-form  $$\tilde Z(E,A)\in\scA_{\C}^{n,n}(\End E)$$ defined in such a way that $$\int_X \tr \left(\tilde Z(E,A)\right) = Z(E).$$

\begin{definition}\cite[Definition 2.20]{DMS}\label{def:z-connection} We say that $A$ is a $Z$-\emph{critical connection} if $$\Ima(e^{-i\phi(E)}\tilde Z(E,A)) = 0,$$ where $\Ima$ denotes (\(-i\) times) the skew-Hermitian component with respect to the Hermitian metric $h$ on $E$.
\end{definition}

We view this as a geometric PDE on a connection on the holomorphic vector bundle $E$, and note that it is nonlinear in the \emph{curvature} of the connection.

\begin{example} \emph{Deformed Hermitian Yang--Mills connections} are precisely  $Z$-critical connections with  $Z(E) = \int_X e^{-i\alpha}\cdot\ch(E)$; more precisely, the resulting equation is \cite[Example 2.25]{DMS} $$\Ima\left(e^{-i\phi(E)} \left (\omega\otimes \Id_{\End E} - \frac{F_A}{2\pi}\right)^n\right) = 0.$$ This equation arose simultaneously in the theoretical physics and mirror symmetry literature in the special case when $E$ is a line bundle \cite{MMMS, LYZ}, and was first suggested by Collins--Yau in higher rank \cite[Section 8.1]{CYarXiv}. For other central charges, the equation is similarly explicit.

\end{example}

\section{Equivariant differential geometry}\label{sec3}
\subsection{Equivariant differential forms}

We recall some basic aspects of the theory of equivariant differential geometry. Let $M$ be a  manifold and let $K$ be a compact Lie group acting on $M$. We let $\mfk$ denote the Lie algebra of $K$, and often use the same notation $v \in \mfk$ for an element of this Lie algebra and its associated vector field on $M$ produced by the $K$-action. We also denote by $$\Omega^*(M) = \bigoplus_{k \geq 0} \Omega^{k}(M)$$ the direct sum over $k$ of the space of $k$-forms on $M$. 

\begin{definition} An \emph{equivariant differential form} is a smooth $K$-equivariant polynomial map $$\alpha: \mfk \to \Omega^*(M),$$ where $K$ acts on $\mfk$ by the adjoint action and on $\Omega^*(M)$ by pullback of differential forms.\end{definition}

The key point of the definition is that there is a natural differential on equivariant differential forms.

\begin{definition}We define the \emph{equivariant differential} of an equivariant differential form $\alpha$ by $$(d_{\eq}\alpha)(v) = d(\alpha(v)) + \iota_v (\alpha(v)).$$We say that   an equivariant differential form  $\alpha$ is \emph{equivariantly closed} if $d_{\eq}\alpha=0$. \end{definition}

Suppose $(M,\omega)$ is a symplectic manifold with a $K$-action preserving $\omega$. Recall that a smooth map $\mu: M \to \mfk^*$ is a \emph{moment map} with respect to $\omega$ if it is $K$-equivariant (with $\mfk^*$ given the coadjoint action) and for all $v \in \mfk$ $$d\langle \mu, v \rangle = -\iota_v \omega,$$ where \(\langle-,-\rangle\) denotes the natural pairing between \(\mathfrak{k}^*\) and \(\mathfrak{k}\). It follows from this definition that the equivariant differential form $$v \mapsto \omega + \langle \mu, v\rangle$$ is equivariantly closed, and in fact this condition is \emph{equivalent} to the moment map condition. We  extend this definition slightly to drop positivity of $\omega$.

\begin{definition} Suppose $M$ is a  manifold with a $K$-action and $\omega$ is a closed $K$-invariant two-form. We call a smooth map $\mu: M \to \mfk^*$ a \emph{moment map} with respect to $\omega$ if it is $K$-equivariant and for all $v \in \mfk$ $$d\langle \mu, v \rangle = -\iota_v \omega.$$
 \end{definition}

We will use two basic properties of equivariantly closed differential forms. The first is analogous to the statement that the wedge product of closed forms is closed, and follows from a straightforward computation.

\begin{lemma}\label{lemma:wedge} If $\alpha$ and $\beta$ are two equivariantly closed differential forms, their wedge product defined by $$(\alpha\wedge\beta)(v) = \alpha(v) \wedge \beta(v)$$ is equivariantly closed.\end{lemma}

The second is analogous to fibre integrals of closed forms being closed.
  
\begin{lemma}\label{lemma:fibre} If $\pi: M \to B$ is a proper submersion between manifolds with $K$-actions making $\pi$ a $K$-equivariant map, then the fibre integral of an equivariantly closed differential form $\alpha$ on $M$, defined by $$\left( \int_{M/B}\alpha\right)(v) = \int_{M/B} \alpha(v),$$ is an equivariantly closed form on $B$.
\end{lemma}

The proof is also a straightforward calculation using that $\pi \circ k = k \circ \pi$ for $k \in K$.

\subsection{Equivariant Chern--Weil theory}\label{sec:equiv-cw}
Suppose now that $E$ is a complex vector bundle of rank $r$ over a manifold $M$, and let $K$ be a compact Lie group acting on $M$ and lifting to $E$. Associated to a connection $A$ on $E$ are Chern--Weil representatives; we now explain the equivariant analogue of Chern--Weil theory. Our main reference is Berline--Getzler--Vergne \cite[Chapter 7]{BGV}. 

We assume that $A$ is $K$-invariant in the sense that for all $k \in K$ $$D_E \circ k^* = k^* \circ D_E.$$ Denote by $F_{A} \in \scA^2(\End E)$ the curvature of the connection $A$ on $E$, which is then $K$-invariant, and let $D_{\End E}: \scA^0(\End E) \to \scA^1(\End E)$ be the induced covariant derivative on $\End E$. 

\begin{definition}\cite[Proposition 7.4]{BGV}\label{def:vb_moment_map}
Let $A$ be a $K$-invariant connection. We say that a $K$-equivariant (smooth) section $\mu$ of $\End E \otimes \mfk^*$, thus associating to each $v \in \mfk$ a section $\langle \mu, v\rangle \in \scA^0(\End E)$,   is a \emph{moment map} for the $K$-action on $E$ if for all $v\in \mfk$ we have $$\iota_vF_A = -D_{\End E}\langle \mu, v\rangle.$$
\end{definition} 

Thus the moment map condition is an equality of elements of $\scA^1(\End E)$. We again emphasise that we do not require positivity (of any kind on $F_A$) in our definition of a moment map  in this higher rank case. 

\begin{definition} Let $\mu \in \scA^0(\End E \otimes \mfk^*)$ be a moment map for the $K$-action on $E$. The \emph{equivariant Chern--Weil representatives} of the \emph{equivariant Chern characters} of $E$ with respect to  $A$ are the equivariant differential forms $\widetilde \ch_{k,\eq}(E,A) \in  \ch_{k,\eq}(E)$ defined, for each $v \in \mfk$, by $$\widetilde \ch_{k,\eq}(E,A)(v) = \tr\left(\frac{1}{k!}\left(\frac{i}{2\pi}\left( F_A + \langle\mu,v\rangle\right)\right)^k\right).$$ \end{definition}

Implicit in this is that the equivariant Chern--Weil representatives thus defined are equivariantly closed \cite[p. 211]{BGV}.

Much as is well-known for line bundles, in the vector bundle setting moment maps can be described explicitly via the infinitesimal action and the connection. For each $v\in \mfk$, let $\L^E_v: \scA^0(E) \to \scA^0(E)$ the infinitesimal action on sections of $E$. Then a moment map for the $K$-action on $E$ is given by \cite[Definition 7.5]{BGV} \begin{equation}\label{moment-map-formula-explicit}\langle\mu, v\rangle = D_v - \L^E_v .\end{equation} Note that neither of $\L^E_v$ nor $D_v$ are individually sections of $\End E$.

\section{$Z$-critical K\"ahler metrics as moment maps}\label{sec4}

Here we consider \(Z\)-critical K{\"a}hler metrics in the finite-dimensional setting. We first compute a moment map for the full curvature tensor of a K{\"a}hler manifold, then use this information to produce equivariant representatives for the equivariant Chern characters of the vertical tangent bundle of a holomorphic submersion. Applying the theory of equivariant differential forms, we prove the finite-dimensional case of Theorem \ref{introthm1}.

\subsection{Chern--Weil theory for the tangent bundle} \label{sec:chern-weil-tangent} 

Consider a K\"ahler manifold $(X,\omega)$, with $X$ not necessarily compact. We assume there is an action of a compact Lie group $K$ on $(X, \omega)$ by Hamiltonian isometries. The K\"ahler metric $\omega$ induces a Hermitian metric \(g\) on the holomorphic tangent bundle $TX^{1,0}$, which in turn induces a (Chern) connection $D$ on $TX^{1,0}$ with curvature $R \in \scA^{1,1}(\End TX^{1,0})$. The connection on $TX^{1,0}$ is $K$-invariant, and our aim is to calculate the resulting equivariant Chern--Weil representatives of the Chern characters of $X$.  To phrase what these Chern--Weil representatives are, we require some further notation. 

\begin{definition}\label{def:tangent-bundle-moment-map}
For a smooth function $h$, we denote by  $$g^{-1}(i\db\d h) \in \scA^{0}(\End TX^{1,0})$$ for the endomorphism of $TX^{1,0}$ obtained by using the Hermitian metric $g$ to raise  the \((0,1)\)-form part of $i\db(\d h)\in \scA^{0,1}(T^*X^{1,0})$ to a \((1,0)\)-vector field. 
\end{definition}

In a little more detail, since $\d h$ is a section of the holomorphic vector bundle $T^*X^{1,0}$, it follows that \(i\db\d h\) is a section of $T^*X^{1,0} \otimes T^*X^{0,1}$, while the Hermitian metric $g$ on  \(TX^{1,0}\) defines an isomorphism $T^*X^{0,1} \cong TX^{1,0}$, thus inducing an isomorphism $$T^*X^{1,0} \otimes T^*X^{0,1} \cong T^*X^{1,0} \otimes  TX^{1,0} \cong \End TX^{1,0},$$ and the image of   $i\db\d h$ under these identifications is what we denote $g^{-1}(i\db\d h).$ In local coordinates, $g^{-1}(i\db\d h)$ may be written \(ig^{\alpha\bar{\beta}}\d_{\bar{\beta}}\d_{\gamma} h\) and  is  skew-Hermitian.

\begin{remark}\label{rem:db-convention}
	We emphasise that the $\db$-operator in Definition \ref{def:tangent-bundle-moment-map} is that of the holomorphic vector bundle $T^*X^{1,0}$, and not the $\db$-operator of differential forms. The operator $g^{-1}$ will always be used with this convention, throughout this work. In particular, we will see in the following proof the emergence of $i\d\db h$, where $\d$ is the $\d$-operator of the anti-holomorphic vector bundle $T^*X^{0,1}$. Unlike the corresponding operators on forms, these operators satisfy $i\d\db h = i\db\d h$ under the identification $T^*X^{1,0}\otimes T^*X^{0,1}\cong T^*X^{0,1}\otimes T^*X^{1,0}$.  
\end{remark}

		\begin{proposition}\label{equivariant-chern-weil} Let $\mu: X \to \mfk^*$ be a moment map for the $K$-action on $(X,\omega)$. Then a moment map \(\sigma\) for the $K$-action on $TX^{1,0}$ is given by $$\langle \sigma, v\rangle = g^{-1}\left (i\bar\partial\partial(\langle \mu, v\rangle)\right),$$ where \(v\in\mathfrak{k}\).

		\end{proposition}
		
		\begin{proof}
		
			We first prove equivariance. Since $\mu$ itself is equivariant, for $k\in K$ and $v\in \mfk$ we have $h_{\mathrm{Ad}(k)v} = k^*h_v,$  where $h_v := \langle \mu, v\rangle$ and \(\mathrm{Ad}:K\to\End(\mathfrak{k})\) denotes the adjoint action of \(K\) on \(\mathfrak{k}\). Further since $K$ acts by biholomorphisms preserving $\omega$, the Hermitian metric $g$ is $K$-invariant and for an arbitrary smooth function $f$ on $X$ we have $$ i\db\d k^*f=  k^*i\db\d  f.$$ It follows that $\sigma^*: \mfk \to \scA^0(\End TX^{1,0})$ defined by $$v \mapsto g^{-1} i\db\d ( \langle \mu, v\rangle)$$ is $K$-equivariant, as required.

	We next fix $v\in\mfk$ and prove the moment map property with respect to $v$. Let $h = \langle \mu, v\rangle $, so that $\iota_v \omega = -dh.$  Recall that by Equation \eqref{moment-map-formula-explicit}, the moment map is given by $$ \left\langle \sigma, v\right\rangle =  D_v - \L_v.$$ On the tangent bundle, the infinitesimal action corresponds with the usual Lie derivative of vector fields. Since the Chern connection agrees with the restriction of the complexified Levi--Civita connection $D^g$ to $TX^{1,0}$, the torsion-free property implies \cite[Example 7.8]{BGV} $$ \left\langle \sigma, v\right\rangle = (D^gv)|_{TX^{1,0}}.$$ 
	
	Write $v = v^{1,0}+v^{0,1}$. Since $D^g$ preserves the decomposition $TX^{\mathbb{C}} = TX^{1,0} \oplus TX^{0,1}$ and $\langle\sigma, v\rangle$ is an endomorphism of $TX^{1,0}$, we must have $(D^g v^{0,1})|_{TX^{1,0}} = 0$, so $$\langle\sigma, v\rangle = (D^g v^{1,0})|_{TX^{1,0}} = D v^{1,0}.$$ As the vector field $v$  is Hamiltonian, $$v^{1,0} = g^{-1}(i\bar \partial h).$$ Since the Chern connection preserves the metric $g$, we have $$Dv^{1,0} =  D\left(g^{-1}(i\bar \partial h)\right) = g^{-1}\bar{D}(i\db h) = g^{-1}(i\partial \bar \partial h) = g^{-1}(i\db\d h),$$ where $\bar{D}$ is the conjugate of the Chern connection, and in the final equality we apply Remark \ref{rem:db-convention}. \qedhere	\end{proof}

\begin{remark}

As noted to us by one of the referees, one may also use the Killing identity in Riemannian geometry \cite[Exercise 1.5]{WB} in order to prove Proposition \ref{equivariant-chern-weil}. One can also prove the result directly using holomorphic normal coordinates; see a previous version of our work \cite[Proposition 4.1]{DH}. \color{black}

\end{remark}

The first notable consequence of this is the following, for which we introduce the notation  $$H_v := g^{-1}(i\db\d (\langle \mu, v\rangle)).$$

	\begin{corollary} The equivariant Chern--Weil representatives of the Chern characters of $TX^{1,0}$ with respect to the Hermitian metric induced by $\omega$ are given for $v \in \mfk$ by $$\widetilde\ch_{k,\eq}(X,\omega)(v) = \tr \left(\frac{1}{k!}\left( \frac{i}{2\pi}(R+ H_v)\right)^k\right).$$ \end{corollary}

This is a standard result for the \emph{first} Chern class, where (since $\tr (iH_v) = -\Delta \langle \mu, v\rangle$ and $\tr (iR) = 2\pi\Ric\omega$) for $j=1$ the result states that the equivariant differential form \begin{equation}\label{convention}v \mapsto \Ric\omega - \frac{1}{2\pi}\Delta \langle \mu, v\rangle \end{equation} is equivariantly closed, and is a representative of \(c_{1,\eq}(X)\). This is, at least beyond equivariance of the form, also a classical consequence of the Bochner formula  \cite[Proposition 8.8.3, Remark 8.8.2]{PG} (having been used many times since \cite[Lemma 28]{GS-duke}, \cite[Section 3]{EL}), and to our knowledge equivariance was first explicitly stated for general compact Lie groups by McCarthy \cite[Proof of Proposition 3.5]{JM-osc}.

Our main applications will instead use a consequence of this global result which holds for the relative tangent bundle of a holomorphic submersion. The setting is the following. We consider a proper holomorphic submersion $\pi: X \to B$ between finite-dimensional K\"ahler manifolds; the fibres are thus compact, but $B$ and $X$ themselves are not required to be so. We denote the fibre of $\pi$ over $b \in B$ by $X_b$. We fix a compact Lie group $K$ acting by biholomorphisms on $X$ and $B$ in such a way that $\pi$ is a $K$-equivariant map. We finally fix a $K$-invariant relatively K\"ahler metric $\omega_X$ along with a moment map $\mu_X: X \to \mfk^*$ with respect to $\omega_X$ (recalling we do not demand that $\omega_X$ be globally positive).

		Letting  \(\mathcal{V}^{1,0}:=\ker(d\pi:TX^{1,0}\to TB^{1,0})\) denote the vertical holomorphic tangent bundle, we write \(g_{\mathcal{V}}\) for the Hermitian metric on \(\mathcal{V}^{1,0}\) induced by the relatively K{\"a}hler metric $\omega_X$ and set  \(R_{\scV}\) to be its curvature, so that $R_{\scV}$ is now a \((1,1)\)-form on \(X\) with values in \(\End(\mathcal{V}^{1,0})\). As before, through the metric $g_{\mathcal{V}}$, for a smooth function $h$ on $X$ we can think of \(i\db_{\mathcal{V}}\d_{\mathcal{V}}h\) as an endomorphism of \(\mathcal{V}^{1,0}\), by raising the \((0,1)\)-form part to a \((1,0)\)-vector field; we will write \(g_{\mathcal{V}}^{-1}i\db_{\mathcal{V}}\d_{\mathcal{V}} h \) for this endomorphism. Here $\partial_{\V}$ is the projection onto $(\V^{1,0})^*$ of $\partial$ and $\bar\partial_{\V}$ is the natural operator on the holomorphic vector bundle $(\V^{1,0})^*$. Noting that $\scV^{1,0}$ admits a natural $K$-action, we may ask for a moment map.
				
		\begin{corollary}\label{equivariant-chern-weil-relative} A moment map \(\sigma\) for the $K$-action on $\scV$ is given by $$\langle \sigma, v\rangle = g_{\mathcal{V}}^{-1}i\db_{\mathcal{V}}\d_{\mathcal{V}} (\langle \mu_X, v\rangle),$$ where \(v\in\mathfrak{k}\).

		\end{corollary}

\begin{proof}

Equivariance of \(\sigma\) is similarly deduced as in Proposition \ref{equivariant-chern-weil}, so we omit the proof of this. The result is local, so we may fix a chart around a point in $B$ and endow this chart with a $K$-invariant  metric $\omega_B$ admitting a moment map $$\mu_B: B \to \mfk^*;$$ such a moment map can always be found, using that $\omega_B$ is exact in the given chart.

While the form $\omega_X$ is only relatively K\"ahler, for $k \gg 0$ the form $$\omega_k = k\omega_B + \omega_X,$$ is genuinely K{\"a}hler, perhaps after shrinking $B$ once more. The associated moment map takes the form $$\mu_k = k\mu_B + \mu_X,$$  and denoting by $R_k$ the curvature of the induced Hermitian metric $g_k$ on $TX^{1,0}$, by Proposition \ref{equivariant-chern-weil} \begin{equation}\label{eq:large-volume-moment-map}
\iota_v R_k = -D_k(g_k^{-1}i\db\d \langle \mu_k, v\rangle ).
\end{equation} This is an equality of $\End TX^{1,0}$-valued one-forms, and we wish to deduce from this the equality of $\End \scV^{1,0}$-valued one-forms \begin{equation}\label{eq:relative-moment-map}
\iota_v R_{\scV} = -D_{\End \scV}(g_{\scV}^{-1}i\db_{\scV}\d_{\scV} \langle \mu_X, v\rangle ).
\end{equation}

To do this, we merely consider the block decomposition of \eqref{eq:large-volume-moment-map} under \(TX^{1,0} = \mathcal{H}^{1,0}\oplus\mathcal{V}^{1,0}\), where \(\mathcal{H}^{1,0}\) is the orthogonal complement of \(\mathcal{V}^{1,0}\) under \(\omega\) (equivalently \(\omega_k\)). Since \(\mathcal{V}^{1,0}\) is a holomorphic subbundle of \(TX^{1,0}\), the component of \(R_k\) in \(\End(\scV^{1,0})\) is just \(R_{\scV^{1,0}}\), as \(\omega_k\) induces the same metric on \(\scV^{1,0}\) as \(\omega\) does for all \(k\). Similarly, the Chern connection of \(TX^{1,0}\) preserves the decomposition \(\mathcal{H}^{1,0}\oplus\mathcal{V}^{1,0}\) and restricts to the Chern connection of \(\scV^{1,0}\). Since the \(\End(\scV^{1,0})\) component of \(g_k^{-1}i\db\d\langle\mu_k, v\rangle\) is equal to \(g_{\scV}^{-1}i\db_{\scV}\d_{\scV}\langle\mu_X, v\rangle\), we immediately obtain \eqref{eq:relative-moment-map}. \qedhere
\end{proof}

To construct the scalar curvature as a  moment map, we will only require the trace of this result. Recall that the \emph{relative anticanonical class} $-K_{X/B} = \Lambda^n \scV^{1,0}$ is by definition the top exterior power of the relative holomorphic tangent bundle; induced from the Hermitian metric $\omega_X$ on $\scV^{1,0}$ is thus a Hermitian metric on $-K_{X/B}$, and we write \(i/2\pi\) times the curvature of this metric as $\rho \in c_1(-K_{X/B})$, so that $\frac{i}{2\pi}\tr (R_{\scV}) = \rho$. Similarly the trace of $g_{\scV}^{-1}\d_{\mathcal{V}}\db_{\mathcal{V}} \langle\mu,v\rangle$ is the vertical Laplacian $\Delta_{\scV} \langle\mu,v\rangle$, defined by restricting $\langle\mu,v\rangle$ to a fibre and taking the Laplacian given by restricting the metric $\omega_X$ to the fibre.
	
	\begin{corollary}\label{relative-ricci-eq} The equivariant differential form $$v \mapsto \rho - \frac{1}{2\pi}\Delta_{\scV} \langle\mu,v\rangle$$ is equivariantly closed, and is a representative of $c_{1,\eq}(-K_{X/B})$.\end{corollary}

\subsection{Scalar curvature as a moment map in finite dimensions}\label{sec:scalarmoment}

Our setup is the same as the end of Section \ref{sec:chern-weil-tangent}, so that $\pi: X \to B$ is a proper holomorphic submersion of relative dimension $n$, $K$ is a compact Lie group acting holomorphically on both $X$ and $B$ in such a way that $\pi$ is $K$-equivariant, and $\omega$ is a $K$-invariant relatively K\"ahler metric on $X$ with cohomology class \(\alpha = [\omega]\). In this generality, we prove a moment map property of the scalar curvature. The novelty in comparison to the prior work (beyond the basic approach) is that we allow \emph{both} the complex structure and the symplectic structure to vary. The results here will be a special case of the results concerning $Z$-critical K\"ahler metrics, but we isolate the scalar curvature result as it is of independent interest. 

We denote by $\omega_b := \omega|_{X_b}$ the resulting  K\"ahler metric on the fibre $X_b$ of $\pi$  for $b\in B$, and suppose that $\mu: X \to \mfk^*$ is a moment map for the $K$-action on $(X, \omega)$. Denote also $$\hat S_b =  n\frac{\int_{X_b}c_1(X_b)\cdot \alpha^{n-1}}{\int_{X_b}\alpha^n}$$ the fibrewise average scalar curvature, which is independent of $b\in B$.

\begin{definition}

We define the \emph{Weil--Petersson form} on $B$ to be the closed $(1,1)$-form $$\Omega = \frac{\hat S_b}{n+1}\int_{X/B}\omega^{n+1} -\int_{X/B} \rho \wedge \omega^n.$$ 

\end{definition}

The Weil--Petersson form is usually only defined for families of cscK manifolds  \cite[Definition 7.1]{FS-moduli} (i.e.\ when $\omega$ has constant scalar curvature when restricted to each fibre), in which case it is known to be semipositive (its strict positivity is also well-understood). 

\begin{theorem}\label{integrable-scalar-curvature}

The map $\sigma: B \to \mfk^*$ defined by $$\langle \sigma(b), v\rangle = \int_{X_b} \langle \mu, v\rangle|_b (\hat S_b - S(\omega_b))\omega_b^n$$ is a moment map for the $K$-action on $(B,\Omega)$.

\end{theorem}

\begin{proof}

Corollary \ref{relative-ricci-eq} proves that the form $$\beta: v \mapsto \rho - \frac{1}{2\pi}\Delta_{\scV} \langle \mu, v\rangle$$ is an equivariantly closed form on $X$. Similarly by hypothesis the form $$\gamma: v \mapsto \omega + \langle \mu, v\rangle$$ is equivariantly closed. It follows that $\beta\wedge \gamma^n$ is equivariantly closed on $X$ and hence its fibre integral $$\int_{X/B} \beta\wedge \gamma^n$$ is an equivariantly closed differential form on $B$.

Unravelling what this means will prove the result. By definition for $v \in \mfk$  \begin{align*} \left(\int_{X/B} \beta\wedge \gamma^n\right)(v) 
&= \left(\int_{X/B} \beta(v)\wedge \gamma(v)^n\right), \\ 
&= \int_{X/B} \left( \rho - \frac{1}{2\pi}\Delta_{\scV} \langle \mu, v\rangle \right)\wedge(\omega + \langle \mu, v\rangle)^n. 
\end{align*} We expand this to obtain \begin{align*}\int_{X/B} \left( \rho - \frac{1}{2\pi}\Delta_{\scV} \langle \mu, v\rangle\right)& \wedge(\omega + \langle \mu, v\rangle)^n \\  &= \int_{X/B}\left( \rho\wedge\omega^{n} - \frac{1}{2\pi}\Delta_{\scV} \langle \mu, v\rangle\omega^n + n\langle \mu, v\rangle\rho\wedge\omega^{n-1}\right),\end{align*} where we use that the fibre integral of forms of degree strictly less than $2n$ vanish. The first of these three terms is a $(1,1)$-form on the base $B$, while for degree reasons the latter two are functions on $B$ which we now calculate. Since $$\left( \int_{X/B} \Delta_{\scV} \langle \mu, v\rangle\omega^n\right)(b) = \int_{X_b}\Delta_b (\langle \mu, v\rangle|_{X_b}) \omega_b^n,$$ and since integrals of functions in the image of the Laplacian vanish, this fibre integral also vanishes. Again working fibrewise, \begin{align*}n\left(\int_{X/B} \langle \mu, v\rangle\rho\wedge\omega^{n-1}\right)(b) &= n\int_{X_b} \langle \mu, v\rangle|_{X_b} \Ric(\omega_b)\wedge\omega_b^{n-1}, \\ &=  \int_{X_b} \langle \mu, v\rangle|_{X_b} S(\omega_b)\omega_b^n,\end{align*} where we use that $\rho|_{X_b} = \Ric(\omega_b)$. 

To conclude we perform a similar, but simpler calculation using $\gamma^{n+1}$, which produces $$\left(\int_{X/B} \gamma^{n+1}\right)(v) = \int_{X/B}\omega^{n+1} + (n+1)\int_{X/B} \langle \mu, v\rangle \omega^n, $$where the latter function satisfies $$\left(\int_{X/B} \langle \mu, v\rangle \omega^n\right)(b) = \int_{X_b} \langle \mu, v\rangle |_{X_b}\omega_b^n.$$ Summing up we conclude that as required the map $\sigma: B \to \mfk^*$ defined by $$\langle \sigma, v\rangle(b) = \int_{X_b} \langle \mu,v\rangle|_b (\hat S_b - S(\omega_b))\omega_b^n$$ is a moment map for the $K$-action on $(B,\Omega)$, by definition of $\Omega$. \end{proof}

\begin{remark} The reason to include the term $\frac{\hat{S}_b}{n+1}\int_{X/B}\omega^{n+1}$ in the Weil--Petersson form is so that the moment map involves the scalar curvature \emph{minus its average}, so that constant scalar curvature metrics are zeroes of the moment map.   \end{remark}

\subsection{The general moment map in finite dimensions}

We next prove a finite-dimensional moment map property for the $Z$-critical operator. Our setup and notation is identical to Section  \ref{sec:scalarmoment}. Our strategy is to first associate to a central charge $Z$ in the sense of Definition \ref{def:central-charge-manifolds} a closed $(1,1)$-form $\Omega_Z$ on $B$. Our central charge is now defined fibrewise, and we denote it by $$Z(X_b,\alpha_b) = \sum_{j,k} a_{jk}\int_{X_b}\alpha_b^{j}\cdot \ch_{k_1}(X_b)\cdot\ldots\cdot \ch_{k_r}(X_b) \in \C,$$ where $a_{jk} \in \C$ are a collection of complex numbers. We note that the value \(Z(X_b,\alpha_b)\) is independent of \(b\in B\), since the fibre integral of a closed form is closed, and the resulting closed form on the base has degree zero, so is a constant function.

We first associate a \emph{complex} $(1,1)$-form $\eta_Z \in \scA^{1,1}(B)$ to the central charge $Z$ through fibre integrals. Our construction will be linear, so we consider a single term $\alpha_b^{j}\cdot \ch_{k_1}(X_b)\cdot\ldots\cdot \ch_{k_r}(X_b)$ of the central charge. The relatively K\"ahler metric $\omega$ induces a Hermitian metric on $\scV^{1,0}$ with curvature $R_{\scV} \in \scA^{1,1}(\End \scV^{1,0})$, and hence induces  Chern--Weil representatives which we denote $$\widetilde \ch_k(\scV^{1,0}, \omega) = \tr\left(\frac{1}{k!}\left(\frac{i}{2\pi}R_{\scV}\right)^k\right) \in  \ch_k(\scV^{1,0}).$$ We associate to this term the closed $(n+1,n+1)$ form on $X$ defined by $$\frac{1}{j+1}\omega^{j+1}\wedge \widetilde \ch_{k_1}(\scV^{1,0}, \omega)\wedge \ldots\wedge\widetilde \ch_{k_r}(\scV^{1,0}, \omega) \in \scA^{n+1,n+1}(X).$$ Taking the fibre integral of this $(n+1,n+1)$-form produces a closed $(1,1)$-form on $B$ associated to the single term $\alpha^{j}\cdot \ch_{k_1}(\scV^{1,0})\cdot\ldots\cdot \ch_{k_r}(\scV^{1,0})$, and extending linearly over all terms of the central charge produces a  closed \emph{complex} $(1,1)$-form $\eta_Z \in \scA^{1,1}(B)$, with complexity arising from the fact that we allow the coefficients $a_{jk}$ to be complex themselves.

\begin{definition}\label{def:omegaz} We define the form $\Omega_Z$ associated to the central charge $Z$ to be $$\Omega_Z = \Ima\left(e^{-i\phi(X_b,\alpha_b)} \eta_Z\right).$$\end{definition}

Here we recall \(\phi(X_b,\alpha_b):= \arg Z(X_b,\alpha_b)\in(-\pi,\pi]\) is the \emph{phase} of the central charge. We will primarily be interested in $\Omega_Z$ rather than $\eta_Z$.

\begin{example}Suppose we take the central charge $Z$ to induce the constant scalar curvature equation, namely $Z(X_b,\alpha_b) = i\int_{X_b} \alpha_b^n -  \int_{X_b} c_1(X_b)\cdot \alpha^{n-1}$. Then \(\Omega_Z\) is a positive multiple of the Weil--Petersson form $$\frac{\hat S_b}{n+1}\int_{X/B}\omega^{n+1} -\int_{X/B} \rho \wedge \omega^n.$$ \end{example} 

Recall the definition of \(\tilde{Z}(X_b, \omega_b)\) from equation \eqref{eq:Z-tilde-def}.

\begin{theorem}\label{thm:fin_dim_manifold_moment_map}
Suppose $\mu: X \to \mfk^* $ is a moment map for the $K$-action on $(X,\omega)$. Then the map $\sigma_Z: B \to \mfk^*$ defined by  $$\langle \sigma_Z(b), v\rangle := \int_{X_b} \langle \mu, v\rangle|_b \Ima(e^{-i\phi(X_b,\alpha_b)}\tilde Z(X_b, \omega_b))\omega_b^n$$ for \(v\in\mathfrak{k}\), is a moment map for the $K$-action on $(B,\Omega_Z)$. 

\end{theorem}

\begin{proof}

Equivariance of the (claimed) moment map will be a formal consequence of the theory of equivariant differential forms and equivariance of the map sending $v$ to  $R_{\scV} + g^{-1}i\db_{\mathcal{V}}\d_{\mathcal{V}} (\langle \mu, v\rangle)$, so we fix $v\in\mfk$ and prove that $$\iota_v{\Omega_Z} = -d  \int_{X_b} \langle \mu, v\rangle|_b \Ima(e^{-i\phi(X_b,\alpha_b) }\tilde Z(X_b, \omega_b))\omega_b^n,$$ with the latter function considered as a function of $b\in B$.

Since $\omega_b^n$ and $ \langle \mu, v\rangle|_b$ are real, it is enough to calculate the moment map for  a single term comprising $\eta_Z$, namely the $(1,1)$-form on $B$ given by $$\frac{1}{j+1}\int_{X/B}\omega^{j+1}\wedge \widetilde \ch_{k_1}(V^{1,0}, \omega)\wedge \ldots\wedge\widetilde \ch_{k_r}(V^{1,0}, \omega).$$ We replace this form with an equivariantly closed form by replacing $\omega$ with $\omega+\mu$ and similarly replacing each term $\widetilde \ch_{k_m}(\scV^{1,0}, \omega)$ with $\widetilde \ch_{k_m,\eq}(\scV^{1,0}, \omega)$, producing by Lemma \ref{lemma:wedge} an equivariantly closed form $$\frac{1}{j+1}(\omega+\mu)^{j+1}\wedge \widetilde \ch_{k_1,\eq}(\scV^{1,0}, \omega)\wedge \ldots\wedge\widetilde \ch_{k_r, \eq}(\scV^{1,0}, \omega) $$ on $X$. Thus by Lemma \ref{lemma:fibre} the fibre integral $$\frac{1}{j+1}\int_{X/B}(\omega+\mu)^{j+1}\wedge \widetilde \ch_{k_1,\eq}(\scV^{1,0}, \omega)\wedge \ldots\wedge\widetilde \ch_{k_r, \eq}(\scV^{1,0}, \omega)  $$ is an equivariantly closed form on $B$, and so, for each $v \in \mfk$, induces a closed $(1,1)$-form plus a function for dimensional reasons.

We calculate this form explicitly to conclude. It is clear that the closed $(1,1)$-form on $B$ must be $$\frac{1}{j+1}\int_{X/B}\omega^{j+1}\wedge \widetilde \ch_{k_1}(\scV^{1,0}, \omega)\wedge \ldots\wedge\widetilde \ch_{k_r}(\scV^{1,0}, \omega); $$ this will lead to the appearance of $\Omega_Z$. To calculate the function  component of the fibre integral, since the  fibre dimension is $n$, the only possible contribution arises from integrating a function against an $(n,n)$-form.

We thus calculate the relevant function. Denote by $h_v = \langle \mu, v\rangle $ and $H_v = g_{\scV}^{-1}i\d_{\mathcal{V}}\db_{\mathcal{V}}h_v \in \scA^{0}(\End\scV^{1,0})$, so that the equivariant Chern--Weil representatives of the Chern characters are given by \begin{align*}\widetilde \ch_{k_m,\eq}(V^{1,0}, \omega)(v) &= \tr\left(\frac{1}{k_m!}\left(\frac{i}{2\pi}(R_{\scV}+H_v)\right)^{k_m}\right), \\ &=  \tr\left(\frac{1}{k_m!}\left(\frac{i}{2\pi}R_{\scV}\right)^{k_m} + \frac{1}{(k_m-1)!}\left(\frac{i}{2\pi}R_{\V}\right)^{k_m-1}\frac{i}{2\pi}H_v+ \O(H_v^2)\right),
\end{align*} where we use that the trace of a commutator vanishes. In order to obtain a function on $B$, only the \emph{linear} term in $H_v$ can contribute; all other terms have too small degree and integrate to zero. Writing $$\tilde\tau_m(\scV^{1,0}, \omega)(v) = \frac{1}{(k_m-1)!}\tr\left(\left(\frac{i}{2\pi}R_{\scV}\right)^{k_m-1}\frac{i}{2\pi}H_v\right),$$  we may discard the $ \O(H_v^2)$-terms, and may replace our equivariant Chern--Weil representative  with $$v\mapsto \widetilde \ch_{k_m}(\scV^{1,0}, \omega)(v) +\tilde \tau_m(\scV^{1,0}, \omega)(v).$$

We wish to replace the section $H_v \in \scA^{0}(\End \scV^{1,0})$ with the function $h_v\in C^{\infty}(X,\R)$ to obtain the actual $Z$-critical equation by using an integration by parts argument. The relevant integral is calculated over a single fibre, so we fix $b \in B$ and argue as follows. Note that $R_{\scV}$ restricts to the fibre $X_b$ as the curvature \(R_b\) of the Hermitian metric induced by $\omega_b$ on $TX^{1,0}_b$. Denote $$\tilde \ell_{m}(X_b,\omega_b) = \frac{1}{j+1}\frac{\omega_b^{j+1}\wedge \widetilde \ch_{k_1}(X_b, \omega_b)\wedge\ldots\wedge \frac{1}{(k_m-1)!}\left(\frac{i}{2\pi}R_b\right)^{k_m-1} \wedge \ldots\wedge\widetilde \ch_{k_r}(X_b,\omega_b)}{\omega_b^n},$$ namely where we replace a single term $\widetilde \ch_{k_m}(X_b,\omega_b)$ with $\frac{1}{(k_m-1)!}\left(\frac{i}{2\pi}R_b\right)^{k_m-1}$. As the numerator is an $\End(TX_b^{1,0})$-valued $(n,n)$-form, overall $\tilde \ell_{m}(X_b,\omega_b)$ is a section of $\End TX_b^{1,0}$.

Consider for the moment an arbitrary section $A \in \scA^0(\End TX_b^{1,0})$. Then we restrict $H_v$ to $X_b$ and calculate on $X_b$ \begin{align}
\mathrm{tr}(AH_v) = \langle A^\flat,i\db_b\d_b h_v\rangle_{g_b}, \nonumber
\end{align} where \(A^\flat\) is \(A\) considered as a section of \(T^*X_b^{1,0}\otimes T^*X_b^{0,1}\) via the metric \(g_b\) and we emphasise that all derivatives are taken on $X_b$.  We then integrate by parts to obtain \begin{align*}\int_{X_b} \tr(AH_v)\omega^n &= \int_{X_b}  \langle A^\flat,i\db_b\d_b h_v\rangle_{g_b}\omega_b^n, \\ &=  \int_{X_b} h_v (i\d_b^*\db_b^*A^{\flat}) {\omega_b^n}.\end{align*}

Thus the integral of interest in calculating the function component of the equivariantly closed form, namely $$\int_{X_b}\omega^{j+1}\wedge \widetilde \ch_{k_1}(\scV^{1,0}, \omega)\wedge\ldots\wedge \tilde  \tau_m(\scV^{1,0}, \omega)(v)\wedge \ldots\wedge\widetilde \ch_{k_r}(\scV^{1,0}, \omega),$$ can be rewritten as $$\int_{X_b}h_v\left(\frac{-1}{2\pi} \d_b^*\db_b^*\left(\tilde \ell_{m}(X_b,\omega_b)^\flat\right)\right) \omega_b^n.$$

From here the calculation involves repeatedly applying this idea and using the definition of the $Z$-critical equation to see that the moment map is indeed given by \[\langle \sigma_Z(b), v\rangle = \int_{X_b} \langle \mu, v\rangle|_b \Ima(e^{-i\phi(X_b,\alpha_b)}\tilde Z(X_b, \omega_b))\omega_b^n.\qedhere\]\end{proof}

\begin{remark}\label{rmk:c-vs-ch} This calculation is why we use Chern characters rather than Chern classes. If one uses Chern classes, following a similar strategy one instead uses the \emph{adjugate} to linearise the Chern--Weil representatives, through $$\frac{d}{d\epsilon}\bigg|_{\epsilon = 0} \det  \left(R+ \epsilon H_v+ t\Id_{\End \scV}\right) = \tr\left(\Adj(R+t\Id_{\scV}) H_v\right);$$ here, the adjugate should be understood as the curvature of the induced Hermitian metric on $\Lambda^{n-1} \scV^{1,0}.$ This produces an equivalent formulation of the $Z$-critical equation, but is notationally more cumbersome.
\end{remark}

\begin{remark}\label{rem:prev_results} A weak moment map property for the special case of the $Z$-critical equation only involving powers of the \emph{first} Chern class is proven in \cite[Section 3.3]{stabilityconditions}. The proof there only applies to isotrivial families (so only a single complex manifold is considered), using very different ideas involving Deligne pairings. \end{remark}

\section{Moment maps on the space of almost complex structures}\label{sec:ac-structures}

The goal of this section is to prove moment map properties for the constant scalar curvature equation and the $Z$-critical equation on the space of almost complex structures. To do so, we begin by developing some geometric aspects of the space of almost complex structures. 

Denoting by $\J(M,\omega)$ the space of almost complex structures on a fixed compact symplectic manifold $(M,\omega)$, we will show that $\J(M,\omega)$ admits a universal family $\pi: \U \to \J(M,\omega)$, so that the fibre over $J\in \J(M,\omega)$ is the almost complex manifold $(M,J)$ endowed with the compatible symplectic form $\omega$, making the fibre into an almost K\"ahler manifold. In particular the universal family has a natural relatively almost K\"ahler metric $\omega_{\U}$, and we will further lift the action of the group of exact symplectomorphisms $\G$ to $\U$, making $$\pi:(\U,\omega_{\U}) \to \J(M,\omega)$$ a $\G$-equivariant almost holomorphic submersion, endowed with a relatively almost K\"ahler metric.

This is precisely the setup of Section \ref{sec4}, beyond the fact that we consider almost complex structures which may not be integrable, and that the spaces involved are infinite-dimensional (though the map $\pi$ has finite relative dimension). We will thus adapt the techniques of Section \ref{sec4} to this more general setting to produce moment map properties for the $Z$-critical equation on the space of almost complex structures. 

\subsection{Preliminaries on the space of almost complex structures}

\subsubsection{The space of almost complex structures}

We recall some aspects of the theory of the space of almost complex structures, referring to Scarpa for a thorough introduction \cite[Section 1.3]{CS}.

	Let \((M,\omega)\) be a  compact symplectic manifold, and recall that an endomorphism $J \in \scA^0(\End TM)$ is an \emph{almost complex structure} if $J^2 = -\Id_{TM}$. 
	
	\begin{definition}

	An almost complex structure \(J\) on \(M\) is \emph{compatible} with \(\omega\) if for all vector fields $X,Y$ on $M$ we have $$\omega(JX,JY) = \omega(X,Y),$$ and $g(X,Y):=\omega(X,JY)$ is a Riemannian metric on \(M\). We denote by \(\mathcal{J}\) or $\J(M,\omega)$  the set of all almost complex structures on \(M\) compatible with \(\omega\).\end{definition}

	The set $\J$ has the natural structure of an infinite-dimensional Fr{\'e}chet manifold \cite[Section 3]{NK} (see also \cite[Section 1.3]{CS}); we will always consider $\J$ with this structure. Going further, the Fr\'echet manifold $\J$ admits itself an almost complex structure, defined in the following manner. Firstly, the tangent space to \(J\in\mathcal{J}\) is naturally identified with the space of smooth endomorphisms \(A\) of \(TM\) satisfying $$AJ+JA=0 \textrm{ and }\omega(AX,JY)+\omega(JX,AY)=0$$ for all vector fields \(X,Y\) on \(M\). We then define the almost complex structure $\mathbb{J}^{\mathcal{J}}$ on $\J$ by setting, for $A \in T_J\J$, \[\mathbb{J}^{\mathcal{J}}_J(A)=JA.\] Koiso further gives $\J$ the structure of a \emph{complex} Fr\'echet manifold, by producing holomorphic Fr\'echet charts on $\J$ (with holomorphic transition functions). In this way, the almost complex structure $\mathbb{J}^{\mathcal{J}}$ is induced by a genuine complex structure on $\J$.

	\subsubsection{The universal family over the space of almost complex structures}
	We next construct a universal family $\U$ over $\J$, whose fibre over $J\in\J$ is the almost complex manifold $(M,J)$. We give the universal family increasing structure in turn. As a smooth Fr\'echet manifold, we define \(\mathcal{U}:=\mathcal{J}\times M\), which we think of as a submersion over \(\mathcal{J}\) with fibre $M$. As an almost complex manifold, we define  \[\mathbb{J}^{\,\mathcal{U}}_{(J,x)}(A,X):=(\mathbb{J}^{\mathcal{J}}_J(A),J_xX).\] The almost complex structure defined in this way is such that the projection  \(\pi:\mathcal{U}\to\mathcal{J}\) is a holomorphic map between almost complex manifolds.
		
	Universal families admit natural \emph{relative} metrics rather than global ones; in our situation this construction is as follows. Consider the  closed $(1,1)$-form $\omega_{\U}$ defined by pulling back $\omega$ from the (smooth) projection $\U \to M$. Then $\omega_{\U}$ is a \emph{relatively almost K{\"a}hler} metric on \(\mathcal{U}\) , in the sense that it is a closed $(1,1)$-form  satisfying $$\omega_{\U}(\mathbb J^{\U}X,\mathbb J^{\U}Y)=\omega_{\U}(X,Y)$$ for tangent vectors $X,Y$ at a point of $\U$, and \(\omega_{\mathcal{U}}(X,\mathbb{J}^{\mathcal{U}}X)>0\) for any non-zero vertical tangent vector \(X\in TM\subset T\mathcal{U}\). In particular, the restriction of $\omega_{\U}$ to any fibre $\pi^{-1}(J)$ of $\pi$ is an almost K{\"a}hler metric on $(M,J)$. Hence on integrable fibres, $\omega_{\U}$ restricts to a genuine K\"ahler metric.

	We next note the existence of a canonical Chern connection $D^{\scV}$ on the vertical (almost) holomorphic tangent bundle \(\mathcal{V}=TM\subset T\mathcal{U}\) over \(\mathcal{U}\), which preserves the Hermitian metric \[\langle X, Y\rangle_{(J,x)} := \omega_x(X, JY),\] and has \((0,1)\)-part equal to the $\db$-operator \[\db^{\scV} X := \db^J X + \db^{\mathcal{H}}X.\] In this second definition, \(\db^J\) is simply the del-bar operator of the almost complex manifold \((M,J)\), and \(\db^{\mathcal{H}}\) is defined by observing the bundle \(\mathcal{V}\) is trivial in the horizontal directions, so we may extend the del-bar operator of \(\mathcal{J}\) to \(\mathcal{V}\).
	
	To construct \(D^{\scV}\), first observe there exists a natural horizontal \(\mathrm{End}(\mathcal{V})\)-valued 1-form \(\Theta\) on \(\mathcal{U}\), defined by \[\Theta_{(J,x)}(A,X):=A_x,\] where we use the canonical identification of \(T_J\mathcal{J}\) with a subspace of endomorphisms of \(\mathcal{V}=TM\). Clearly \(\Theta\) is complex linear, so we may consider it as a horizontal \((1,0)\)-form with values in \(\End(\scV)\).
	
	Next, recall any compatible almost complex structure \(J\) on \((M, \omega)\) induces a natural Chern connection on \(TM\) defined by \begin{equation}\label{chern-connection}D^J_XY:=D_X^gY-\frac{1}{2}J(D_X^gJ)Y,\end{equation} where \(g\) is the Hermitian metric defined by \(J\) and \(\omega\), and \(D^g\) is the Levi-Civita connection of \(g\). Using these data, we can define the connection \(D^{\scV}\) on \(\mathcal{V}\) by \[D^{\scV}:=D^J+d^{\mathcal{J}}-\frac{1}{2}\mathbb{J}^{\scV}\Theta.\] To elaborate, in the vertical direction at \(J\), the connection is given by \(D^J\). Since \(\mathcal{V}\) is trivial in the horizontal directions, the horizontal part of any connection is \(d^{\mathcal{J}}\) plus a horizontal \(\mathrm{End}(\mathcal{V})\)-valued 1-form; we have chosen \(-\frac{1}{2}\mathbb{J}^{\scV}\Theta\) for that 1-form, where \(\mathbb{J}^{\scV}\in\End(\scV)\) denotes the almost complex structure of \(\mathcal{V}\). 
	
	In fact, since \(\End(\scV)\) is trivial in the horizontal directions, we may take the horizontal derivative of \(\mathbb{J}^{\scV}\) to get \(d^{\mathcal{J}}\mathbb{J}^{\scV}\in\scA^0(\mathcal{H}^*\otimes\End(\scV))\). It is immediate that \(\Theta = d^{\mathcal{J}}\mathbb{J}^{\scV}\), since the tangent space to \(\mathcal{J}\) consists exactly of infinitesimal changes in \(\omega\)-compatible almost complex structures.

\begin{lemma}
\(D^{\scV}\) is the Chern connection associated to the natural del-bar operator and Hermitian metric on \(\mathcal{V}\).
\end{lemma}

\begin{proof}
Since \(D^J\) is already the Chern connection of \((M,J)\), we only need to compute the conditions in the horizontal direction. To see that \(D^{\scV}\) is Hermitian in the horizontal direction, take sections \(s_1\) and \(s_2\) of \(\mathcal{V}\) that are constant in the horizontal direction. Then for any horizontal tangent vector \(A\), \begin{align*}
d\langle s_1,s_2\rangle(A) &= \frac{1}{2}d^{\mathcal{J}}(\omega(s_1,\mathbb{J}^{\scV}s_2)+\omega(s_2,\mathbb{J}^{\scV}s_1))(A) \\
&= \frac{1}{2}\left( \omega(s_1, (d^{\mathcal{J}}\mathbb{J}^{\scV})(A)s_2)+\omega(s_2,(d^{\mathcal{J}}\mathbb{J}^{\scV}(A))s_1) \right) \\
&= \omega\left(s_1, \frac{1}{2}\Theta(A)s_2\right)+\omega\left(s_2,\frac{1}{2}\Theta(A)s_1\right) \\
&= \omega\left(s_1,\mathbb{J}^{\scV}\left(d^{\mathcal{J}}-\frac{1}{2}\mathbb{J}^{\scV}\Theta\right)s_2\right)(A) + \omega\left(s_2,\mathbb{J}^{\scV}\left(d^{\mathcal{J}}-\frac{1}{2}\mathbb{J}^{\scV}\Theta\right)s_1\right)(A) \\
&= \langle s_1, D^{\scV}_As_2\rangle +  \langle D^{\scV}_As_1,s_2\rangle.
\end{align*} We have already remarked that \(\Theta\) is complex linear, so is considered as a \((1,0)\)-form, from which it follows that the \((0,1)\)-part of \(D^{\scV}\) is the $\db$-operator of \(\mathcal{V}\).
\end{proof}

	\subsubsection{The action of the group of exact symplectomorphisms} We now construct an action of the group of exact symplectomorphisms $\G$ on $\U$. As a smooth manifold $\U$ is given by $\J \times M$, and for $\gamma \in \G$ we define $\gamma(J,x) = (\gamma\cdot J, \gamma(x)),$ where $\G$ acts on $\J$ by $$(\gamma \cdot J)(u) = \gamma_*\circ J\circ \gamma^*(u),$$ and where pullback of tangent vectors is defined since $\gamma$ is a diffeomorphism.  The map $\U\to \J$ is thus $\G$-equivariant, and the form $\omega_{\U}$ on $\U$ is $\G$-invariant.
	
	We next calculate the infinitesimal action, to show that this is a holomorphic action. The Lie algebra of $\G$ is naturally identified with the space $C^\infty_0(M,\mathbb{R})$ of smooth functions on $M$ with integral zero. The vector field on $\J$ associated with a function \(f\in C^\infty_0(M,\mathbb{R})\) is defined as follows. Firstly, \(f\) generates a Hamiltonian vector field \(X_f\) on \(M\) defined by $$df=\omega(-,X_f).$$ Let \(\Phi^f_t\) be the time-\(t\) flow of \(X_f\), which is a diffeomorphism of \(M\) to itself. Using this flow, we can define a vector field \(A_f\) on \(\mathcal{J}\) by \[A_f|_J:=\left.\frac{d}{dt}\right|_{t=0}(\Phi_{-t}^f)^*J=-\mathcal{L}_{X_f}J.\] The vector field on $\U$ associated to \(f\in C^\infty_0(M,\mathbb{R})\) is then $A_f + X_f$.
	
	\begin{lemma}\label{lem:hol_action}
	The action of $\G$ on $\U$ is holomorphic.
	\end{lemma}
	
	\begin{proof}
	We must show the vector field $A_f + X_f$ on $\U$ generated by  \(f\in C^\infty_0(M,\mathbb{R})\) is holomorphic. In other words, we must show that the Lie derivative of \(\mathbb{J}^{\mathcal{U}}=\mathbb{J}^{\mathcal{J}}+\mathbb{J}^{\scV}\) along this vector field vanishes. We calculate \begin{align*}
	&\mathcal{L}_{A_f+X_f}(\mathbb{J}^{\mathcal{J}}+\mathbb{J}^{\scV}), \\
	=& \mathcal{L}_{A_f}(\mathbb{J}^{\mathcal{J}})+\mathcal{L}_{A_f}(\mathbb{J}^{\scV})+\mathcal{L}_{X_f}(\mathbb{J}^{\mathcal{J}})+\mathcal{L}_{X_f}(\mathbb{J}^{\scV}),  \\
	=& \mathcal{L}_{A_f}(\mathbb{J}^{\scV})+\mathcal{L}_{X_f}(\mathbb{J}^{\scV}).
	\end{align*} Here we use that the action of \(\mathcal{G}\) on \(\mathcal{J}\) is holomorphic to obtain \(\mathcal{L}_{A_f}(\mathbb{J}^{\mathcal{J}})=0\) (using the definition of the infinitesimal action and the complex structure $\mathbb{J}^{\mathcal{J}}$), and that \(\mathbb{J}^{\mathcal{J}}\) is constant in the vertical directions to obtain \(\mathcal{L}_{X_f}(\mathbb{J}^{\mathcal{J}})=0\).
	
	Now, \(A_f\) is the vector field associated to the one-parameter family of diffeomorphisms \(\Psi_t^f\) of \(\J\), where \[\Psi_t^f(J):=(\Phi_{-t}^f)^*J.\] Hence \[\mathcal{L}_{A_f}(\mathbb{J}^{\scV})|_J:=\frac{d}{dt}\bigg |_{t=0}\mathbb{J}^{\scV}|_{(\Phi_{-t}^f)^*J}=\frac{d}{dt}\bigg |_{t=0}(\Phi_{-t}^f)^*J=-\mathcal{L}_{X_f}J=-\mathcal{L}_{X_f}\mathbb{J}^{\scV}|_J.\] Thus \(\mathcal{L}_{A_f}(\mathbb{J}^{\scV})+\mathcal{L}_{X_f}(\mathbb{J}^{\scV})=0\), proving the result.	
	\end{proof}

Thus both the $\bar\partial$-operator and Hermitian metric (induced by $\omega_{\U}$) on $\scV$ are compatible with the $\G$-action. It follows that the Chern connection, and hence its curvature, are also $\G$-invariant.

\subsection{Equivariant differential geometry of almost K\"ahler manifolds}

Our next goal is to understand the equivariant representatives of the Chern characters of the almost holomorphic tangent bundle of an almost K\"ahler manifold. We will then apply a similar idea to Corollary \ref{equivariant-chern-weil-relative} in order to obtain a version of this statement for an almost holomorphic submersion, where we only assume that the relative dimension is finite. These will then be applied to the almost holomorphic submersion $\pi:(\U,\omega_{\U}) \to \J(M,\omega)$.

We begin by considering an almost K\"ahler manifold $(M,J,\omega)$, endowed with an almost holomorphic action of a compact Lie group $K$, which fixes $\omega$ and admits a moment map $\mu: M \to \mfk^*$. Denote by $TM^{1,0}$ the almost holomorphic tangent bundle, which admits a Chern connection defined by Equation \eqref{chern-connection}. We let $D^{TM}, D^{TM_{\C}}$ and $D$ denote the Chern connections on the smooth tangent bundle, complexified tangent bundle and almost holomorphic tangent bundle respectively. Denote by $g$ the Hermitian metric on $TM^{1,0}$ induced by the almost K\"ahler metric $\omega$.

\begin{proposition}\label{prop:almost-kahler-moment} A moment map for the $K$-action on $TM^{1,0}$ is given by $$\langle \sigma, v\rangle = g^{-1}(i\db\d \langle \mu, v\rangle),$$ where $\db:\scA^0(T^*M^{1,0})\to\scA^{0,1}(T^*M^{1,0})$ is the $\db$-operator of $T^*M^{1,0}$.

\end{proposition}

\begin{remark}\label{rem:almost-hol-db-convention}
Unlike in the K{\"a}hler case, we are somewhat forced to deal with the $\d$ and $\db$ operators on differential forms, since these are the more fundamental operators in this setting. In particular, the operator $$\db:\scA^0(T^*M^{1,0})\to\scA^{0,1}(T^*M^{1,0})$$ of Proposition \ref{prop:almost-kahler-moment} is defined as $$\db \alpha := p(\pi_{(1,1)}d\alpha),$$  where $\pi_{(1,1)}$ is projection to the $(1,1)$-forms, and $p$ is the projection $$(T^*M^{1,0}\otimes T^*M^{0,1})\oplus (T^*M^{0,1}\otimes T^*M^{1,0})\to T^*M^{0,1}\otimes T^*M^{1,0};$$ note that $(\pi_{(1,1)}\circ d)|_{\Omega^{(1,0)}(M)}$ is the $\db$-operator on exterior forms.

The $\d$-operator on $T^*M^{0,1}$ is defined similarly. These definitions will again imply $g^{-1}(i\d\db h) = g^{-1}(i\db\d h)$, as per Remark \ref{rem:db-convention} in the K{\"a}hler case; one need simply bear in mind that the operators $\d$ and $\db$ here produce sections of the appropriate tensor product bundles, and not exterior forms.
\end{remark}

\begin{proof}

In comparison with Proposition \ref{equivariant-chern-weil},  we do not assume that $J$ is integrable, hence the Chern connection on $TM$ no longer agrees with the Levi--Civita connection. In particular, the Chern connection has torsion, which is given by the Nijenhuis tensor: \begin{equation}\label{torsion}D^{TM}_XY - D^{TM}_YX  - [X,Y] = -\frac{1}{4}N_J(X,Y),\end{equation} where $N_J(X,Y) = [X,Y] + J([JX, Y] + [X,JY]) - [JX,JY].$ The $K$-action has infinitesimal lift to $TM$ given by the Lie derivative, meaning $\langle\sigma^{TM},v\rangle = D^{TM}_v - \L_v $ by Equation \ref{moment-map-formula-explicit}, where $\sigma^{TM}$ is the moment map for the $K$-action on $TM$. It follows that $$\left\langle\sigma^{TM},v\right\rangle = D^{TM}v - \frac{1}{4}\iota_vN_J,$$ where we view $\iota_vN_J = N_J(v,-)$ as a section of $\End TM$.

The Chern connection on $TM$ induces the Chern connections on $TM_{\C}$ and on $TM^{1,0}$; since the Chern connection preserves $J$, it preserves types of forms, meaning that $(R^{TM_{\C}})^{\End TM^{1,0}}$, the $\End TM^{1,0}$-component of the curvature, is simply the curvature of the Chern connection on $TM^{1,0}$. Thus the moment map for the Chern connection on $TM^{1,0}$ is given by $$\langle \sigma, v\rangle = \left(D^{TM_{\C}}v - \frac{1}{4}\iota_vN_J\right)^{\End TM^{1,0}}.$$ 

Note that the $\End TM^{1,0}$-component of $\iota_vN_J$ actually vanishes: writing $v = v^{1,0} + v^{0,1}$, for $Y$ a section of $TM^{1,0}$ the Nijenhuis tensor satisfies $N_J(v^{1,0},Y)=[v^{1,0},Y]^{(0,1)} \in \scA^0(TM^{0,1})$, while $N_J(v^{0,1},Y)$ vanishes in general. For the remaining term, since again the Chern connection preserves type, $(D^{TM_{\C}}v)^{\End TM^{1,0}} = Dv^{1,0}$. Thus the moment map $\mu$ is simply $$\langle \mu, v\rangle = Dv^{1,0},$$ viewed through the induced section  of $\End TM^{1,0}$ (noting that for $X\in\scA^0(TM^{1,0})$, that $D_Xv^{1,0} \in \scA^0(TM^{1,0})$ holds also in the almost K\"ahler setting).

We next use our assumption that  $v = J\nabla h$ for $h = \langle \mu, v\rangle$, so that $v^{1,0} = g^{-1}(i\bar\partial h)$. We claim that $Dv^{1,0} = D(g^{-1}(i\bar\partial h))= g^{-1}(i\partial\bar\partial \langle \mu, v\rangle)$, similarly to the integrable case. Since the Chern connection is compatible with the metric, we have $$ Dg^{-1}(i\bar\partial h) = g^{-1}\left(\bar{D}(i\bar\partial h)\right),$$ however relating the Chern connection on forms to the differential now involves torsion, hence the Nijenhuis tensor. We thus go through the argument in detail. 

As we view $ g^{-1}\left(\bar{D}(i\bar\partial h)\right)$ as inducing a section of $\End TM^{1,0}$, we may fix sections $X,Y\in \scA^0(TM^{1,0})$ and calculate 
\begin{align*}
g(D_Yg^{-1}(i\db h),\bar{X}) &= g(g^{-1}(\overline{D_{\bar{Y}}(-i\d h)}), \bar{X}) \\
&= \langle \overline{D_{\bar{Y}}(-i\d h)}, \bar{X}\rangle \\
&= \overline{\langle D_{\bar{Y}}(-i\d h), X\rangle}.
\end{align*}
By Equation \ref{torsion} and a general formula relating the connection to the differential for connections with torsion, for a general one-form $\beta$ and sections $X,Y$ of $TM^{1,0}$ we have $$(d\beta)(X,\bar Y) =  \left\langle D_X \beta, \bar Y\right\rangle - \left\langle D_{\bar Y} \beta, X\right\rangle -\frac{1}{4} \beta(N_J(X,\bar Y)).$$ In our situation $\beta = -i\d h$, so this implies $$\langle D_{\bar{Y}}(-i\bar\partial h), X\rangle = (id\partial h)(X,\bar Y),$$ since $N_J(X,\bar Y)$ and $\langle D_X\partial h, \bar{Y}\rangle$ both vanish. Since we take $X,Y\in \scA^0(TM^{1,0})$, we see $$(id\partial h)(X,\bar Y) = (i\db\d h)(X,\bar Y),$$ and so the induced sections of $\End TM^{1,0}$ satisfy $$g^{-1}\left(D(i\bar\partial h)\right) =g^{-1}(i\partial\bar\partial h)=g^{-1}(i\db\d h),$$ as per Remark \ref{rem:almost-hol-db-convention}.\end{proof}

For $M$ finite-dimensional, we thus obtain that the Chern--Weil representatives of the Chern characters of the almost holomorphic tangent bundle are given by $$\tilde \ch_j(TM^{1,0},\omega): v\to \tr\left(\frac{1}{j!}\left(\frac{i}{2\pi}(R + i\bar\partial\partial\langle\mu, v\rangle )\right)^j\right)$$ with $R$ the curvature of the Chern connection on $TM^{1,0}$; finite-dimensionality is used to take the trace, and hence to make sense of Chern--Weil theory. For the first Chern character, this produces the Chern--Ricci curvature, and this moment map equation for the Chern--Ricci curvature  has previously been proven by Lejmi using other techniques \cite[Lemma 2.6]{ML}. 

For applications, we require a relative version of this setup, where the relative dimension is finite but the dimensions of the base and total space are infinite. We thus consider an almost holomorphic submersion $\pi: (\U,\omega_{\U}) \to B$, with an action of a (possibly infinite-dimensional) Lie group $\mathcal G$ on $\U$ and $B$ making $\pi$ equivariant and such that $\omega_{\U}$ is a $\mathcal G$-invariant relatively almost K\"ahler metric, along with a moment map $\mu: \U \to \Lie \mathcal G^*$. We in addition assume that $B$ has an auxiliary $\G$-invariant almost K\"ahler metric.

\begin{remark}\label{aside}
Our main interest is in the space of almost complex structures $\J(M,\omega)$, which is known to admit a K\"ahler metric which is invariant under the action of the group of exact symplectomorphisms, meaning our arguments apply to this setting. In general, as our argument is local, if $\mathcal G$ is finite-dimensional we may choose a compatible almost K\"ahler metric in a given Fr\'echet chart in $B$ and average over the $\mathcal G$-action. In particular, if $\mathcal G$ is finite-dimensional, we do not need to assume $B$ admits an almost K\"ahler metric globally.
\end{remark}

We let $\V = (\ker d\pi) \subset T\U$, and let $\V^{1,0}$ denote the vertical almost holomorphic tangent bundle. The form $\omega_{\U}$ induces a Hermitian metric $g_{\V}$, and hence a Chern connection, on $\V^{1,0}$. We may thus ask for a moment map for the $\mathcal G$-action on $\V^{1,0}$ with respect to this connection.

\begin{proposition} \label{almost-kahler-CW-relative}
A moment map $\sigma$ for the $\mathcal G$-action on $\V^{1,0}$ is given by $$\langle \sigma, v\rangle = g_{\V}^{-1} i\db_{\V}\d_{\V} (\langle\mu, v\rangle).$$
\end{proposition}

\begin{proof}

The moment map equality $$\iota_v R = -D \langle \sigma, v\rangle$$ is an equality of sections of $\scA^1(\End \V^{1,0})$, so to prove it we may choose a point $b \in B$, a tangent vector lying in $T_bB$ and an element of the fibre of $\End \V^{1,0}$ over $b$. Thus although $B$ may be infinite-dimensional, we may choose a $k\gg 0$ such that $\omega_{\U}+k\pi^*{\omega_B}$ defines a Hermitian metric on $T\U^{1,0}$ in all directions that enter into the proof (noting that infinite-dimensionality prevents this argument over all of $\U$).  Thus the same argument as Corollary \ref{equivariant-chern-weil-relative} applies to prove our result, using Proposition \ref{prop:almost-kahler-moment}. \end{proof}
\subsection{Moment maps on the space of almost complex structures}  Our next goal is to prove that the $Z$-critical equation appears as a moment map on the \emph{infinite-dimensional} space $\J(M,\omega)$, in particular recovering the moment map property of the scalar curvature. We denote by $Z(M,\alpha)$ a fixed central charge; as its value is topological, we consider this as fixed independently of $J \in \J(M,\omega)$. 

In order to state our general moment map property, we extend the definition of the $Z$-critical equation to the almost K\"ahler setting: the appropriate definition is identical to the integrable setting, namely Definition \ref{defpde}, where we use the Hermitian metric on the almost holomorphic tangent bundle induced by the almost K\"ahler metric $\omega$ to produce the Chern connection, and hence to take curvature.

\begin{theorem}\label{thm:infinite_dim_manifold_moment_map} The map $\sigma:\J(M,\omega) \to C^{\infty}_0(M,\R)^*$ defined by $$\langle\sigma(J),f\rangle := \int_M f\Ima(e^{-i\phi(M,\alpha)} \tilde Z(M,J))\,\omega^n$$ is a moment map for the $\G$-action on $(\J(M,\omega), \Omega_Z)$. 
\end{theorem}

Here the closed $(1,1)$-form $\Omega_Z$ on $\J(M,\omega)$ is induced by the central charge $Z$ and defined as a fibre integral in the same way as Definition \ref{def:omegaz}. The group $\G$ is the group of exact symplectomorphisms of $(M,\omega)$. 

As a special case, that we recover the moment map property of the scalar curvature in the integrable setting follows from the equality of Weil--Petersson form we use and the $L^2$-metric on $\J(M,\omega)$ \cite[Theorem 10.3]{FS-moduli} (see also Fujiki \cite[Theorem 8.2]{AF}). We note that it is not currently known that the Weil--Petersson form we use agrees with the Donaldson--Fujiki form in the non-integrable setting, though we expect this to be the case.

\begin{proof}[Proof of Theorem \ref{thm:infinite_dim_manifold_moment_map}]

	The equivariance condition of the moment map  follows from the same arguments as the equivariance part of Theorem \ref{thm:fin_dim_manifold_moment_map}. Thus we must only prove the moment map equation. The equivariant Chern--Weil representatives of the relative tangent bundle $\V^{1,0}\subset T\U^{1,0}$ are provided by Proposition \ref{almost-kahler-CW-relative}. This is the key step of the proof in the integrable case, namely Theorem \ref{thm:fin_dim_manifold_moment_map}, and the rest of the argument goes through in the same manner.	 \end{proof}

This proof also applies to general families of almost K\"ahler manifolds over a finite-dimensional base (whose almost complex structure may not be integrable), much as in the integrable case of Theorem \ref{thm:fin_dim_manifold_moment_map}. We use the analogous notation to that setting, so that $(X,\omega) \to B$ is a $K$-equivariant almost holomorphic submersion between finite-dimensional almost complex manifolds, with $\omega\in \alpha$ a $K$-invariant relatively almost K\"ahler metric with moment map $\mu: X \to \mfk^*$, and $Z$ is a choice of central charge inducing a closed form $\Omega$ on $B$ (not necessarily of type $(1,1)$) and with phase $\phi_b(X_b,\alpha_b)$. $\tilde Z(X_b,\omega_b)$ is then computed on the fibre $(X_b,\omega_b)$ as throughout.

\begin{theorem}
The map $\sigma_Z: B \to \mfk^*$ defined by  $$\langle \sigma_Z(b), v\rangle := \int_{X_b} \langle \mu, v\rangle|_b \Ima(e^{-i\phi(X_b,\alpha_b)}\tilde Z(X_b, \omega_b))\omega_b^n$$ for \(v\in\mathfrak{k}\), is a moment map for the $K$-action on $(B,\Omega_Z)$. \end{theorem}
        
		\begin{remark}
		
		 This moment map result also applies to families over general infinite-dimensional bases, as long as one can apply Remark \ref{aside}, namely if the base admits a $\mathcal G$-invariant almost K\"ahler metric. As a special case, the moment map interpretation of the Hermitian scalar curvature also applies to families of almost K\"ahler manifolds over arbitrary finite-dimensional almost complex bases, generalising Theorem \ref{integrable-scalar-curvature}.
		\end{remark}
		
\section{$Z$-critical connections as moment maps}\label{sec:Z-critical-connections}
\subsection{The space of unitary connections}

We now turn to the vector bundle setting, where our first aim is to describe the geometry of the space of unitary connections on a fixed Hermitian vector bundle over a compact  complex manifold; we denote by $\scA(E,h)$ the set of unitary connections on the Hermitian vector bundle  $(E,h)$  over $X$. A reference for the basic theory of the space of unitary connections is Sektnan \cite[Chapter 2]{LS}.

We wish to endow $\scA(E,h)$  with further structure; proofs of all results stated can be found in \cite[Chapter 2]{LS}. Firstly, the space of unitary connections is an \emph{affine space} modelled on $\Omega^1(\End_{SH}(E,h)),$ where $\End_{SH}(E,h)$ denotes skew-Hermitian endomorphisms. This is a real affine space on which we construct an almost complex structure as follows. At any connection $A \in\scA(E,h)$ the tangent space at $A$ satisfies $T_{A} \scA(E,h) \cong \Omega^1(\End_{SH}(E,h)),$ and we define a  complex structure by decomposing an element $a \in T_{A} \scA(E,h)$ as $a = a^{1,0} + a^{0,1}$ and setting $$\mathbb J(a) = -ia^{1,0} + ia^{0,1}.$$  One checks that $\mathbb J(a) \in \Omega^1(\End_{SH}(E,h))$ and that $\mathbb J^2 = -\Id$, defining an almost complex structure on $\scA(E,h)$. One further checks that the almost complex structure is formally integrable in the sense that its Nijenhuis tensor vanishes.

We next define the analogue of the universal family over $\scA(E,h),$ which is a Hermitian complex vector bundle with a unitary universal connection. Consider the product manifold $\scA(E,h) \times X$. Pulling back the Hermitian vector bundle $(E,h)$ gives a Hermitian vector bundle $(\E,h_{\E})$ over $\scA(E,h) \times X$. We endow $(\E,h_{\E})$ with a universal unitary connection $D_{\E}$ using the splitting of the tangent bundle on  $\scA(E,h) \times X$ as follows. For $(u,v) \in T_{(A, x)}(\scA(E,h) \times X)$ and \(s\in\scA^0(\E)\) we set $$(D_{\E}s)(u,v) = (ds)(u) + (D_A s)(v),$$ where \(d\) is the usual exterior derivative on $\scA(E,h)$ extended to the bundle \(\mathcal{E}\) which is trivial in the \(\scA(E,h)\) directions. Thus on every submanifold $\{A\}\times X \subset \scA(E,h) \times X,$  the connection restricts to the connection $A$ on $E$ itself.  The connection \(D_{\E}\) is easily seen to be unitary, since each $A$ is unitary and the Hermitian metric on $\E$ is the pullback of  the Hermitian metric $h$ on $E$ itself.
 
Denote by $\G$ the \emph{gauge group} of \((E,h)\), namely the group of sections of $\End E$ which are unitary on each fibre. The Lie algebra of the gauge group is  given by $\Lie \G = \End_{SH}(E,h),$ namely skew-Hermitian endomorphisms of $E$. For $f \in \G$, the action of the gauge group on $\scA(E,h)$ is given by $$f\cdot A = f^{-1}\circ D_A \circ f.$$ Since the gauge group $\G$ acts on $E$, there is naturally an induced action on $\E$ in such a way that $\E$ is a $\G$-equivariant vector bundle on $\scA(E,h) \times X$, where $X$ is given the trivial $\G$-action. The Hermitian metric and the connection $D_{\E}$ on $\E$ are then $\G$-invariant.

\subsection{The Hermitian Yang--Mills condition as a moment map}

Using the universal vector bundle $\E\to \scA(E,h) \times X$, we next give a new proof of the moment map interpretation of the Hermitian Yang--Mills equation  (originally due to Atiyah--Bott \cite[Section 9]{AB}, see also \cite[Section 4]{SD-surfaces}). We begin by constructing a K\"ahler metric $\Omega$ on $\scA(E,h)$. For this, denote by $F_{D_{\E}} \in \scA^2(\E)$ the curvature of the universal connection $D_{\E}$ on $\E$.

\begin{definition} We define $\Omega \in \scA^{1,1}(\scA(E,h))$ to be the fibre integral $$\Omega = \int_{\scA(E,h) \times X/\scA(E,h)} \tr\left(\frac{1}{2!}\left(\frac{i}{2\pi}F_{D_{\E}}\right)^2\right) \wedge \omega^{n-1}.$$\end{definition}

Thus from general properties of fibre integrals, $\Omega$ is a closed, $\G$-invariant $(1,1)$-form on $\scA(E,h)$, with $\G$-invariance following from $\G$-invariance of the connection $D_{\E}$ itself. We next prove positivity.

\begin{proposition}\label{curvature-bundle} $\Omega$ is a K\"ahler metric satisfying $$\Omega(a,b) = -\frac{1}{8\pi^2}\int_X \tr(a\wedge b)\omega^{n-1}$$ for all \(a,b\in T_A\mathcal{A}(E,h)\cong\Omega^1(\End_{SH}(E,h))\). \end{proposition}

\begin{proof}

We first consider in more detail the curvature $F_{D_{\E}}$ of the connection $D_{\E}$. This is a $2$-form on $\scA(E,h) \times X$ with values in $\End \E$, and so we can decompose it using the splitting of the tangent bundle $T(\scA(E,h) \times X) \cong T\scA(E,h) \oplus TX$. Viewing $\scA(E,h) \times X$ as a submersion over $\scA(E,h)$ with fibre $X$, we decompose  $F_{D_{\E}}$ into purely horizontal, purely vertical and mixed components with respect to the splitting.

The purely horizontal component of \(F_{D_{\E}}\) vanishes, since \(D_{\E}\) is simply the exterior derivative in the horizontal directions. Since the vertical component of \(D_{\E}\) at \(A\in\mathcal{A}(E,h)\) is just \(D_A\), the purely vertical component of \(F_{D_{\E}}\) is \(F_A\) on the fibre over \(A\). Finally, the mixed component is computed as the horizontal exterior derivative of the connection \(D_A\). It is straightforward to see that this mixed component thus takes a tangent vector \(a\in\scA(E,h)\), and maps it to the vertical \(\End_{SH}(E,h)\)-valued 1-form given by identifying tangent vectors to \(\mathcal{A}(E,h)\) with \(\Omega^1(X,\End_{SH}(E,h))\).

Finally, in order to compute \(\Omega\), note that the fibre integrand must include two horizontal components in order not to vanish on the base. Thus, only the mixed components of the curvature are integrated. Since these take horizontal tangent vectors as input and map them to the corresponding elements of \(\Omega^1(X,\End E)\), we have \begin{align*}
	\Omega(a,b) &= -\frac{1}{8\pi^2}\int_{\scA(E,h) \times X/\scA(E,h)} \tr\left(F_{D_{\E}}(a,-) \wedge F_{D_{\E}}(b,-)\right) \wedge \omega^{n-1} \\
	&= -\frac{1}{8\pi^2}\int_{X} \tr\left(a \wedge b\right) \wedge \omega^{n-1}.\qedhere
\end{align*} \end{proof}
 
 \begin{remark} This argument also shows that $\Omega$ agrees with the K\"ahler metric used by Atiyah--Bott and Donaldson on $\scA(E,h)$; this has been shown by Donaldson in the case that $\omega \in c_1(L)$ for an ample line bundle $L$ using different ideas \cite[Proposition 14]{SD-infinite}. 
 \end{remark}

The next step is to calculate the \emph{equivariant} curvature of the universal connection on $\E$. We thus fix an element $e \in \scA^0(\End_{SH} (E,h)) = \Lie \G$, and we wish to calculate the associated section \(\tilde{e}\) of $\End \E$ such that $$\iota_e F_{D_{\E}} = -D_{\End \E}\tilde e.$$

\begin{proposition} The Hamiltonian for the curvature of $F_{D_{\E}}$  with respect to $e \in \scA^0(\End_{SH} (E,h)) = \Lie \G$ is $-e$, so that $$\iota_eF_{D_{\E}} = -D_{\End \E} \left(-e\right).$$\end{proposition}

\begin{proof}
	
	In the statement of the proposition we have identified the element \(e\) of the Lie algebra with its fundamental vector field. For clarity here we distinguish the two and denote by $v_e$ the vector field on $\scA(E,h)$ induced by $e$. The tangent space to $A$ in $A$ is given by $\Omega^1(\End_{SH}(E,h))$, and the value of the vector field $v_e$ at the point $A$  is $D_Ae$ (where \(D_A\) is extended to \(\End E\)). 
	
	We have seen that the contraction \(\iota_{v_e}F_{D_{\mathcal{E}}}\) of the curvature with a horizontal vector field converts the horizontal vector field to its associated vertical 1-form in \(\Omega^1(X,\End_{SH}(E,h))\). In particular, on the fibre over \(A\), \[\iota_{v_e}F_{D_{\mathcal{E}}}=D_Ae=d(e) + D_A e = D_{\End \E}e,\] where we used that \(e\) is constant in the horizontal directions. \end{proof}

The theory of equivariant curvature then produces the following.

\begin{corollary} 
For $e \in \scA^0(\End_{SH} (E,h))$ set $$\langle \mu, e\rangle = -e.$$ Then the equivariant Chern--Weil representatives of the Chern characters of $\E$ are for $e\in \Lie\G$ given by $$\widetilde \ch_{k,\eq}(\E,D_{\mathcal{E}})(e) = \tr\left(\frac{1}{k!} \left(\frac{i}{2\pi} (F_{D_{\E}} +\langle \mu, e\rangle)\right)^k\right).$$

\end{corollary}

We can now give a new, geometric proof of the moment map interpretation of the Hermitian Yang--Mills equation. 

\begin{theorem} A moment map for the $\G$-action on $(\scA(E,h), \Omega)$ is given by $$\nu: \scA(E,h) \to \scA^0(\End_{SH} (E,h))^*,$$ where $$\langle\nu(A),e\rangle = \frac{1}{4\pi^2n}\int_X \tr\left( e \left( \Lambda_{\omega}F_{A} + 2\pi i\lambda \Id_E\right)\right)\omega^n, $$ and \(\lambda:=n\deg(E)/\rk(E)\).
\end{theorem}

\begin{proof}
 
Consider the equivariantly closed form on $\scA(E,h)\times X$ defined by $$\tr\left(\frac{1}{2!}\left(\frac{i}{2\pi}(F_{D_{\E}}+ \langle \mu, e\rangle)\right)^2\right)\wedge \omega^{n-1}.$$ Its fibre integral $$\int_{\scA(E,h)\times X/\scA(E,h)} \tr\left(\frac{1}{2!}\left(\frac{i}{2\pi}(F_{D_{\E}}+ \langle \mu, e\rangle)\right)^2\right)\wedge \omega^{n-1}$$ is an equivariantly closed form on $\scA(E,h)$, and is hence a two-form plus a function. By Proposition \ref{curvature-bundle} the 2-form is $$\Omega(a,b) = -\frac{1}{8\pi^2}\int_{X} \tr\left(a \wedge b\right) \wedge \omega^{n-1},$$ while the function is given by \begin{equation}\label{eq:initial_moment_map}
A \mapsto  \int_X  \tr\left(\frac{1}{4\pi^2}F_{A}e\right)\wedge\omega^{n-1};
\end{equation} we note that the term involving $e^2$ vanishes under the fibre integral. 

Now, note we may produce another equivariant form on the base by integrating over the fibres a constant multiple of the form \[\tr\left(\frac{i}{2\pi}(F_{D_\mathcal{E}}+\langle\mu,e\rangle)\right)\wedge\omega^n.\] Since \(\omega\) is purely vertical and the curvature \(F_{D_\mathcal{E}}\) has no purely horizontal component, the 2-form component of the fibre integral vanishes, and we are left with the function \[-\frac{i}{2\pi}\int_X\tr(e)\omega^n = -\frac{i}{2\pi}\int_X\tr(e\Id_E)\omega^n.\] Adding on a constant multiple of this to the function \eqref{eq:initial_moment_map}, we produce the moment map \[\langle\nu(A),e\rangle = \frac{1}{4\pi^2n}\int_X \tr\left( e \left( \Lambda_{\omega}F_{A} + 2\pi i\lambda \Id_E\right)\right)\omega^n. \qedhere \] \end{proof}

\subsection{$Z$-critical connections}

The proof in the general setting of $Z$-critical connections is similar. Recall in this setting that a \emph{central charge} takes the form $$Z(E) = \sum_{j=0}^n \int_X \rho_j\alpha^j\cdot\ch(E)\cdot \Theta,$$ where $\ch(E)$ is the total Chern character of $E$, the $\rho_j\in \C$ are complex numbers and $\Theta \in \bigoplus_l H^{l,l}(X,\C)$ is an auxiliary cohomology class.

We associate to the central charge a closed $(1,1)$-form $\Omega_Z$ on $\scA(E,h)$. The process will be linear, and so we first explain how to associate a closed $(1,1)$-form to a single  one of the terms $$\int_X \alpha^{j}\cdot \ch_{k}(E)\cdot \Theta_{n-j-k}$$ comprising the central charge. Recall that we fix a closed, complex differential form $\theta \in \Theta$ and a K\"ahler metric $\omega \in \alpha$. We then associate to this term the $(1,1)$-form on $\scA(E,h)$ defined by the fibre integral \begin{equation}\label{bundle-fibre-int-z}
\int_{\scA(E,h) \times X/\scA(E,h)} \omega^{j}\wedge \tr\left(\frac{1}{(k+1)!}\left(\frac{i}{2\pi}F_{D_{\E}}\right)^{k+1}\right)\wedge\theta_{n-j-k}.
\end{equation} By linearity we obtain $\eta_Z \in \scA^{1,1}(\scA(E,h))$ and we set $$\Omega_Z = \Ima(e^{-i\phi(E)} \eta_Z ) \in  \scA^{1,1}(\scA(E,h)).$$

\begin{theorem}\label{connection-thm-infinite}
The moment map for the $\G$-action on $(\scA(E,h), \Omega_Z)$ is given by $$\nu: \scA(E,h) \to \scA^0(\End_{SH} (E,h))^*,$$ where $$\langle \nu(A),e\rangle = -\frac{i}{2\pi}\int_X \tr(e \Ima (e^{-i\phi(E)}\tilde Z(E,A)))\omega^n. $$
\end{theorem}

\begin{proof}

The proof is similar to the previous one. We replace the differential form of interest with $$\int_{\scA(E,h) \times X/\scA(E,h)} \omega^{j}\wedge \tr\left(\frac{1}{(k+1)!}\left(\frac{i}{2\pi}(F_{D_{\E}}+\langle\mu,e\rangle)\right)^{k+1}\right)\wedge\theta_{n-j-k};$$ this equivariant differential form is equivariantly closed on $\scA(E,h)$. Expanding shows that this is the 2-form \eqref{bundle-fibre-int-z} plus the function $$A \mapsto -\frac{i}{2\pi}\int_X \tr\left(e\,\omega^j\wedge \frac{1}{k!}\left(\frac{i}{2\pi}F_{A}\right)^k \wedge \theta_{n-j-k}\right).$$ Taking a linear combination, multiplying by \(e^{-i\varphi(E)}\), then taking the imaginary part finishes the proof.\end{proof}

\subsection{Finite dimensions}

Although our proofs in the vector bundle setting have directly proven moment map properties on the infinite-dimensional space $\scA(E,h)$, the proofs go through in a similar manner in the finite-dimensional case and hence we only give the statement. We require holomorphic variation, which is more transparent considering Hermitian metrics on holomorphic vector bundles rather than connections directly. Thus consider an almost holomorphic vector bundle $\E \to B \times X$ of rank $r$, with $X$ a compact almost K\"ahler manifold, such that there is a holomorphic action of a compact Lie group $K$ on $B\times X$, lifting to $\E$ and hence making $\E$ a $K$-equivariant vector bundle, and acting trivially on $X$. We view $\E$ as a family of almost holomorphic vector bundles over $X$ parametrised by $B$.

We suppose $h$ is a $K$-invariant Hermitian metric on $\E$ with Chern connection $D_{\E}$. Define a closed $(1,1)$-form $\Omega_Z$ on $B$ associated to the central charge $Z$ by linearity in each component and attaching the form \begin{equation*}
\int_{B \times X/B} \omega^{j}\wedge \tr\left(\frac{1}{(k+1)!}\left(\frac{i}{2\pi}F_{D_{\E}}\right)^{k+1}\right)\wedge\theta_{n-j-k}
\end{equation*}
to the appropriate term of \(Z\).  As before, $\omega \in \alpha$ is an almost K\"ahler metric and $\theta \in \Theta$ is a closed complex differential form, both taken to be fixed. 
 Suppose in addition that $\mu$ is a moment map for the curvature $F_{D_{\E}}$, so that in particular for all $v \in \mfk$ $$\iota_v F_{D_{\E}} = -D_{\End\E}\langle \mu, v\rangle$$ where $\langle\mu, v\rangle \in \scA^0(\End\E).$

\begin{theorem}\label{connection-thm-finite}
The moment map for the $K$-action on $(B, \Omega_Z)$ is given by $$\nu:B \to \mfk^*,$$ where $$\langle\nu(b), v\rangle = \frac{i}{2\pi}\int_X \tr(\langle \mu_b, v\rangle \Ima (e^{-i\phi(\E_b)}\tilde Z(\E_b,A_b)))\omega^n, $$ where $ \Ima (e^{-i\phi(\E_b)}\tilde Z(\E_b,A_b))$ is defined fibrewise and \(\E_b,A_b,\mu_b\) denote the restrictions of \(\E,A,\mu\) to the fibre \(\{b\}\times X\).
\end{theorem}

Note that we require the $K$-action on $X$ to be trivial, so that $\omega$ and $\theta$ pull back to \emph{equivariantly closed} forms on $B \times X$; this is used in the proof. 

\bibliographystyle{alpha}
\bibliography{bib}

\end{document}